\newif\ifarxiv\arxivfalse
\arxivtrue

\ifarxiv
    \documentclass{article}
\else
    \documentclass{svjour3}
    \smartqed 
\fi

\usepackage[utf8]{inputenc}
\usepackage{amsmath,amssymb,amsthm,amscd,graphicx,enumerate}
\numberwithin{equation}{section}
\usepackage{geometry}
\usepackage{authblk}
\usepackage{xargs}
\usepackage[linesnumbered]{algorithm2e}
\usepackage[most]{tcolorbox}
\usepackage[caption=false]{subfig}
\usepackage{float}
\usepackage{mathtools}

\DeclarePairedDelimiter\floor{\lfloor}{\rfloor}

\tcbset{colback=yellow!10!white, colframe=black, 
        highlight math style= {enhanced, 
            colframe=red,colback=red!10!white,boxsep=0pt}
       }
 
 \DeclareMathOperator*{\argmin}{arg\,min}
 \newcommand\innprod[2]{\left\langle{}#1{},{}#2{}\right\rangle}%
 \newcommand\func[3]{#1:#2\rightarrow#3}
 \newcommand\E{\mathbb E}%
 \newcommand\R{\mathbb R}%
 \newcommand\A{\mathcal{A}_H^{p}}%
 \newcommand\PS{\psi_{\bar x, H}^p}%
 \newcommand\br{\beta_\rho}
 \newcommand\prox{\mathrm{prox}_{F/H}^p}%
 \newcommand\dom{\mathrm{dom}}%
 \newcommand\dist{\mathrm{dist}}%
 \newcommand\set[1]{\left\{{}#1{}\right\}} 
 \newcommandx\seq[3][{2=k\geq 0},{3={}}]{\{#1\}_{#2}^{#3}}
 
 \newtheorem{thm}{Theorem}
 \numberwithin{thm}{section}
 \newtheorem{lem}[thm]{Lemma}

 \newtheorem{exa}[thm]{Example}
 \newtheorem{defin}[thm]{Definition}
 
 \makeatletter
    \@namedef{subjclassname@2020}{%
      \textup{2020} Mathematics Subject Classification}
 \makeatother

\ifarxiv
    \title{{\bf High-order methods beyond the classical complexity bounds, I: inexact high-order proximal-point methods}}
    \author{
    {Masoud Ahookhosh\footnote{Department of Mathematics, University of Antwerp, Middelheimlaan 1, B-2020 Antwerp, Belgium.
			E-mail: masoud.ahookhosh@uantwerp.be}} \qquad  {Yurii Nesterov\footnote{Center for Operations Research and Econometrics (CORE) and Department of Mathematical Engineering (INMA), Catholic University of Louvain (UCL), 34voie du Roman Pays, 1348 Louvain-la-Neuve, Belgium. E-mail: Yurii.Nesterov@uclouvain.be\\ This project has received funding from the European Research Council (ERC) under the European Union's Horizon 2020 research and innovation programme (grant agreement No. 788368).}}
    }
    \providecommand{\keywords}[1]{\textbf{Keywords:} #1}
    
\begin{document}
    \maketitle
     \begin{abstract}
         In this paper, we introduce a \textit{Bi-level OPTimization} (BiOPT) framework for minimizing the sum of two convex functions, where both can be nonsmooth. The BiOPT framework involves two levels of methodologies. At the upper level of BiOPT, we first regularize the objective by a $(p+1)$th-order proximal term and then develop the generic inexact high-order proximal-point scheme and its acceleration using the standard estimation sequence technique. At the lower level, we solve the corresponding $p$th-order proximal auxiliary problem inexactly either by one iteration of the $p$th-order tensor method or by a lower-order non-Euclidean composite gradient scheme with the complexity $\mathcal{O}(\log \tfrac{1}{\varepsilon})$, for the accuracy parameter $\varepsilon>0$. Ultimately, if the accelerated proximal-point method is applied at the upper level, and the auxiliary problem is handled by a non-Euclidean composite gradient scheme, then we end up with a $2q$-order method with the convergence rate $\mathcal{O}(k^{-(p+1)})$, for $q=\floor{p/2}$, where $k$ is the iteration counter. 
     \end{abstract}
    
     \keywords{Convex composite optimization, High-order proximal-point operator, Bi-level optimization framework, Lower complexity bounds, Optimal methods, Superfast methods}
    
 \else
    \begin{document}
    \journalname{Mathematical Programming}
    \title{High-order methods beyond the classical complexity bounds, I: inexact high-order proximal-point methods}
    \author{%
		Masoud Ahookhosh \and
		Yurii Nesterov%
	}
	\institute{
	M. Ahookhosh
		\at
			Department of Mathematics, University of Antwerp, Middelheimlaan 1, B-2020 Antwerp, Belgium.\\
			\email{masoud.ahookhosh@uantwerp.be}%
	\and 
	Y. Nesterov 
	    \at
              Center for Operations Research and Econometrics (CORE) and Department of Mathematical Engineering (INMA), Catholic University of Louvain (UCL), 34voie du Roman Pays, 1348 Louvain-la-Neuve, Belgium. \\
              \email{Yurii.Nesterov@uclouvain.be} \\          %
              This project has received funding from the European Research Council (ERC) under the European Union's Horizon 2020 research and innovation programme (grant agreement No. 788368).
    }
 \date{}
    \maketitle
    \begin{abstract}
        In this paper, we introduce a \textit{Bi-level OPTimization} (BiOPT) framework for minimizing the sum of two convex functions, where both can be nonsmooth. The BiOPT framework involves two levels of methodologies. At the upper level of BiOPT, we first regularize the objective by a $(p+1)$th-order proximal term and then develop the generic inexact high-order proximal-point scheme and its acceleration using the standard estimation sequence technique. At the lower level, we solve the corresponding $p$th-order proximal auxiliary problem inexactly either by one iteration of the $p$th-order tensor method or by a lower-order non-Euclidean composite gradient scheme with the complexity $\mathcal{O}(\log \tfrac{1}{\varepsilon})$, for the accuracy parameter $\varepsilon>0$. Ultimately, if the accelerated proximal-point method is applied at the upper level, and the auxiliary problem is handled by a non-Euclidean composite gradient scheme, then we end up with a $2q$-order method with the convergence rate $\mathcal{O}(k^{-(p+1)})$, for $q=\floor{p/2}$, where $k$ is the iteration counter. 
        \keywords{Convex composite optimization\and High-order proximal-point operator\and Bi-level optimization framework\and Lower complexity bounds\and Optimal methods\and Superfast methods}
    \end{abstract}
\fi

\section{Introduction}\label{sec:intro}

\paragraph{{\bf Motivation.}}\label{sec:motiv}
Central to the entire discipline of convex optimization is the concept of complexity analysis for evaluating the efficiency of a wide spectrum of algorithms dealing with such problems; see  \cite{nemirovsky1983problem,nesterov2018lectures}. For example, under the Lipschitz smoothness of the gradient of the objective function, the fastest convergence rate for first-order methods is of order $\mathcal{O}(k^{-2})$ for the iteration counter $k$; cf. \cite{ahookhosh2017optimal,ahookhosh2018solving,nesterov2005smooth,nesterov2013gradient}. Likewise, if the objective is twice differentiable with Lipschitz continuous Hessian, the best complexity for second-order methods is of order $\mathcal{O}(k^{-7/2})$; see \cite{arjevani2019oracle}. In the recent years, there is an increasing interest to applying high-order methods for both convex and nonconvex problems; see, e.g., \cite{agarwal2018lower,arjevani2019oracle,birgin2017worst,gasnikov2019optimal,jiang2019optimal}. If the objective is $p$-times differentiable with Lipschitz continuous $p$th derivatives, then the fastest convergence rate for $p$th-order methods is of order $\mathcal{O}(k^{-(3p+1)/2})$; cf. \cite{arjevani2019oracle}. 

In general, for convex problems, the classical setting involves a one-to-one correspondence between the methods and problem classes. In other words, there exists and unimprovable complexity bound for a class of methods applied to a class of problems. In fact, under the Lipschitz (H\"older) continuity of the $p$th derivatives, the $p$th-order methods is called {\it optimal} if it attains the convergence rate $\mathcal{O}(k^{-(3p+1)/2})$, and if a method attains a faster convergence rate (under stronger assumptions than the optimal methods), we call it {\it superfast}. For example, first-order methods with the convergence rate $\mathcal{O}(k^{-2})$ and second-order methods with the convergence rate $\mathcal{O}(k^{-7/2})$ are optimal under the Lipschitz (H\"older) continuity of the first and the second derivatives, respectively. Recently, in \cite{nesterov2020superfast}, a {\it superfast second-order method} with the convergence rate $\mathcal{O}(k^{-4})$ has been presented, which is faster than the classical lower bound $\mathcal{O}(k^{-7/2})$. The latter method consists of an implementation of a third-order tensor method where its auxiliary problem is handled by a Bregman gradient method requiring second-order oracles, i.e., this scheme is implemented as a second-order method. We note that this method assumes the Lipschitz continuity of third derivatives while the classical second-order methods apply to problems with Lipschitz continuous Hessian. This clearly explains that the convergence rate $\mathcal{O}(k^{-4})$ for this method is not a contradiction with classical complexity theory for second-order methods. 

One of the classical methods for solving optimization problems is the {\it proximal-point} method that is given by
\begin{equation}\label{eq:proxMethod}
    x_{k+1} = \argmin_{x} \set{h(x)+\tfrac{1}{2\lambda}\|x-x_k\|^2},
\end{equation}
for the function $h(\cdot)$, a given point $x_k$, and $\lambda>0$. The first appearance of this algorithm dated back to 1970 in the works of Martinet \cite{martinet1970breve,martinet1972determination}, which is further studied by Rockafellar \cite{rockafellar1976monotone} when $\lambda$ is replaced by a sequence of positive numbers $\seq{\lambda_k}$. Since its first presentation, this algorithm has been subject of great interest in both Euclidean and non-Euclidean settings, and many extensions has been proposed; for example see \cite{ahookhosh2019bregman,bauschke2016descent,bolte2018first,guler1992newproximal,iusem1994entropy,teboulle1992entropic}. 

Recently, Nesterov in \cite{nesterov2020inexact} proposed a {\it bi-level unconstrained minimization} (BLUM) framework by defining a novel high-order proximal-point operator using a $p$th-order regularization term 
\begin{equation*}\label{eq:prox0}
    \mathrm{prox}_{h/H}^p(\bar x)=\argmin_{x\in E} \set{h(x)+\tfrac{H}{p+1} \|x-\bar x\|^{p+1}},
\end{equation*}
see Section \ref{sec:inexact} for more details. This framework consists of two levels, where the upper level involves a scheme using the high-order proximal-point operator, and the lower-level is a scheme for solving the corresponding proximal-point minimization inexactly. Therefore, one has a freedom of choosing the order $p$ of the proximal-point operator and can also choose a proper method to approximate the solution of the proximal-point auxiliary problem. Applying this framework to twice smooth unconstrained problems with $p=3$, using an accelerated third-order method at the upper level, and solving the auxiliary problem by a Bregman gradient method lead to a second-order method with the convergence rate $\mathcal{O}(k^{-4})$.  The main goal of this paper is to extend the results of \cite{nesterov2020inexact} onto the composite case.

\vspace{-4mm}
\subsection{Content}\label{sec:content}
In this paper, we introduce a {\it Bi-level OPTimization} (BiOPT) framework that is an extension of the BLUM framework (see \cite{nesterov2020inexact}) for the convex composite minimization. 
In our setting, the objective function is the sum of two convex functions, where both of them can be nonsmooth. At the first step, we regularize the objective function by a power of the Euclidean norm $\|\cdot\|^{p+1}$ with $p\geq 1$, following the same vein as \eqref{eq:proxMethod}. The resulted mapping is called {\it high-order proximal-point operator}, which is assumed to be minimized approximately at a reasonable cost. If the first function in our composite objective is smooth enough, in Section~\ref{sec:inexact}, we show that this auxiliary problem can be inexactly solved by one step of the $p$th-order tensor method (see Section \ref{sec:ptensor}). Afterwards, we show that the plain proximal-point method attains the convergence rate $\mathcal{O}(k^{-p})$ (see Section \ref{sec:ihopp}),  while its accelerated counterpart obtains the convergence rate $\mathcal{O}(k^{-(p+1)})$ (see Section \ref{sec:aihopp}). 

We next present our bi-level optimization framework in Section~\ref{sec:biLevel}, which opens up entirely new ground for developing highly efficient algorithms for simple constrained and composite minimization problems. In the upper level, we can choose the order $p$ of the proximal-point operator and apply both plain and accelerated proximal-point schemes using the estimation sequence technique. We then assume that the differentiable part of the proximal-point objective is smooth and strongly convex relative to some scaling function (see \cite{bauschke2016descent,lu2018relatively}) and then design a non-Euclidean composite gradient algorithm using a Bregman distance to solve this auxiliary problem inexactly. It is shown that the latter algorithm will be stopped after $\mathcal{O}(\log \tfrac{1}{\varepsilon})$ of iterations, for the accuracy parameter $\varepsilon>0$. Hence, choosing a lower-order scaling function for the Bregman distance, there is a possibility to apply lower-order schemes for solving the auxiliary problem that will lead to lower-order methods in our convex composite setting. 

Following our BiOPT framework, we finally pick a constant $p$ for the $p$th-order proximal-point operator and apply the accelerated method to the composite problem at the upper level. Then, we introduce a high-order scaling function and show that the differentiable part of the proximal-point objective is $L$-smooth and $\mu$-strongly convex relative to this scaling function, for $L, \mu>0$. We consequently apply the non-Euclidean composite gradient method to the auxiliary problem, which only needs the $p$th-order oracle for even $p$ and the $(p-1)$th-order oracle for odd $p$. Therefore, we end up with a high-order method with the convergence rate of order $\mathcal{O}(k^{-(p+1)})$ under some suitable assumptions. We emphasize while this convergence rate is faster than the classical lower bound $\mathcal{O}(k^{-(3p-2)/2})$ for $p=3$, it is sub-optimal for other choices of $p$. However, we show that our method can overpass the classical optimal rates for some class of structured problems. We finally deliver some conclusion in Section~\ref{sec:conclusion}.

\subsection{Notation and generalities} \label{sec:notation}
In what follows, we denote by $\E$ a finite-dimensional real vector space and by $\E^*$ its dual spaced composed by linear functions on $\E$. For such a function $s\in \E^*$, we denote by $\innprod{s}{x}$ its value at $x\in \E$.

Let us measure distances in $\E$ and $\E^*$ in a Euclidean norm. For that, using a self-adjoint positive-definite operator $\func{B}{\E}{\E^*}$ (notation $B=B^* \succ 0$), we define
\begin{align*}
    \|x\|=\innprod{Bx}{x}^{1/2},\quad x\in \E,\quad \|g\|_{*} = \innprod{g}{B^{-1}g}^{1/2}, \quad g\in \E^*.
\end{align*}
Sometimes, it will be convenient to treat $x\in \E$ as a linear operator from $\R$ to $\E$, and $x^*$ as a linear operator from $\E^*$ to $\R$. In this case, $xx^*$ is a linear operator from $\E^*$ to $\E$, acting as follows:
\begin{align*}
    (xx^*)g = \innprod{g}{x} x\in\E,\quad g\in\E^*.
\end{align*}
For a smooth function $\func{f}{\E}{\R}$ denote by $\nabla f(x)$ its gradient, and by $\nabla^2 f(x)$ its Hessian evaluated at the point $x\in \E$. Note that
\begin{align*}
    \nabla f(x)\in\E^*, \quad \nabla^2 f(x)h\in\E^*, \quad x,h\in\E.
\end{align*}
We denote by $\ell_{\bar x}(\cdot)$ the linear model of convex function $f(\cdot)$ at point $\bar x\in\E$ given by
\begin{equation}
    \ell_{\bar x}(x) = f(\bar x)+\innprod{\nabla f(\bar x)}{x-\bar x}, \quad x\in\E.
\end{equation}
Using the above norm, we can define the standard Euclidean prox-functions
\begin{align*}
    d_{p+1}(x)=\tfrac{1}{p+1}\|x\|^{p+1}, \quad x\in\E.
\end{align*}
where $p\geq 1$ is an integer parameter. These functions have the following derivatives:
\begin{equation}\label{eq:dp1Der}
    \begin{array}{rcl}
        \nabla d_{p+1}(x) &= &\|x\|^{p-1}Bx, \quad x\in\E, \\
        \nabla^2 d_{p+1}(x) &= &\|x\|^{p-1}B+(p-1)\|x\|^{p-3}Bxx^*B \succeq \|x\|^{p-1}B. 
    \end{array}
\end{equation}
Note that function $d_{p+1}(\cdot)$ is uniformly convex (see, for example, \cite[Lemma 4.2.3]{nesterov2018lectures}):
\begin{equation}\label{eq:dp1LowerBound}
    d_{p+1}(y) \geq d_{p+1}(x)+\innprod{d_{p+1}(x)}{y-x}+\tfrac{1}{p+1} \left(\tfrac{1}{2}\right)^{p-1} \|y-x\|^{p+1}, \quad x,y\in \E.
\end{equation}

In what follows, we often work with directional derivatives. For $p\geq 1$, denote by 
\begin{align*}
    D^pf(x)[h_1,\ldots,h_p]
\end{align*}
the directional derivative of function $f$ at $x$ along directions $h_i\in\E$, $i= 1,...,p$. Note that $D^pf(x)[\cdot]$ is a symmetric $p$-linear form. Its norm is defined in a standard way:
\begin{equation}\label{eq:dpfhiNorm}
    \|D^pf(x)\| = \max_{h_1,\ldots,h_p} \set{\left|D^pf(x)[h_1,\ldots,h_p]\right| ~:~ \|h_i\|\leq 1,~ i=1,\ldots,p}.
\end{equation}
If all directions $h_1,\ldots,h_p$ are the same, we apply the notation
\begin{align*}
    D^pf(x)[h]^p,\quad h\in\E.
\end{align*}
Note that, in general, we have (see, for example, \cite[Appendix 1]{nesterov1994interior})
\begin{equation}\label{eq:dpfhpNorm}
    \|D^pf(x)\| = \max_{h} \set{\left|D^pf(x)[h]^p\right| ~:~ \|h\|\leq 1}.
\end{equation}
In this paper, we work with functions from the problem classes $\mathcal{F}_p$, which are convex and $p$ times continuously differentiable on $\E$. Denote by $M_p(f)$ its uniform upper bound for its $p$th derivative:
\begin{equation}\label{eq:mpf}
    M_p(f)=\sup_{x\in\E} \|D^pf(x)\|.
\end{equation}

\section{Inexact high-order proximal-point methods}\label{sec:inexact}
Let function $\func{f}{\E}{\R}$ be closed convex and possibly non-differentiable and let $\func{\psi}{\E}{\R}$ be a simple closed convex function such that $\dom \psi\subseteq \mathrm{int}(\dom f)$. We now consider the convex composite minimization problem
\begin{equation}\label{eq:prob}
    \min_{x\in\dom \psi}~\set{F(x)=f(x)+\psi(x)},
\end{equation}
where it is assumed that \eqref{eq:prob} has at least one optimal solution $x^*\in\dom \psi$ and $F^*=F(x^*)$. This class of problems is general enough to cover many practical problems from many application fields like signal and image processing, machine learning, statistics, and so on. In particular, for the simple closed convex set $Q\subseteq\E$, the simple constrained problem
\begin{equation}\label{eq:conProb}
    \begin{array}{ll}
        \min & f(x) \\
        \mathrm{s.t.} & x\in Q
    \end{array}
\end{equation}
can be rewritten in the form \eqref{eq:prob}, i.e.,
\begin{equation}\label{eq:compProb}
        \min_{x\in\dom \psi}~ f(x)+\delta_Q(x),
\end{equation}
where $\delta_Q(\cdot)$ is the indicator function of the set $Q$ given by
\begin{align*}
    \delta_Q(x)=\left\{
    \begin{array}{ll}
        0       &~~\mathrm{if}~x\in Q,  \\
        +\infty &~~\mathrm{if}~x\not\in Q.
    \end{array}
    \right.
\end{align*}

Let us define the \textit{$p$th-order composite proximal-point operator}
\begin{equation}\label{eq:prox}
    \prox(\bar x)=\argmin_{x\in\dom \psi} \set{\PS(x)=f(x)+\psi(x)+H d_{p+1}(x-\bar x)},
\end{equation}
for $H>0$ and $p\geq 1$, which is an extension of the $p$th-order proximal-point operator given in \cite{nesterov2020inexact}. Moreover, if $p=1$, it reduces to the classical proximal operator
\begin{align*}
    \prox\bar x) = \argmin_{x\in\dom \psi} \set{f(x)+\psi(x)+\tfrac{H}{2} \|x-\bar x\|^2}.
\end{align*}
Our main objective is to investigate the global rate of convergence of high-order proximal-point methods in accelerated and non-accelerated forms, where we approximate the proximal-point operator \eqref{eq:prox} and study the complexity of such approximation. To this end, let us introduce the set of \textit{acceptable solutions} of \eqref{eq:prox} by
\begin{equation}\label{eq:A}
     \A(\bar x,\beta)= \set{(x,g)\in\dom \psi\times\E^* ~:~ 
g\in\partial \psi(x), \; \|\nabla f_{\bar x, H}^p(x)+g\|_*\leq \beta \|\nabla f(x)+g\|_*},
\end{equation}
where
\begin{equation}\label{eq:f}
    f_{\bar x, H}^p(x)=f(x)+H d_{p+1}(x-\bar x),
\end{equation}
where $\beta\in[0,1)$ is the tolerance parameter. 
Note that if $\psi\equiv 0$, then the set $ \A(\bar x,\beta)$ leads to {\it inexact acceptable solutions} for the problem \eqref{eq:prox}, which was recently studied for smooth convex problems in \cite{nesterov2020inexact}. 
Let us emphasize that extending the definition of inexact acceptable solutions from \cite{nesterov2020inexact} for nonsmooth functions is not a trivial task because not all subgradients $g\in\partial \psi(x)$ satisfy the inequality \eqref{eq:A}. In the more general setting of the composite minimization, we address this issue in Section~\ref{sec:nepgm} using a non-Euclidean composite gradient scheme that suggests which subgradient $g\in\partial \psi(x)\neq \emptyset$ can be explicitly used in \eqref{eq:A}.

Since function $F(\cdot)$ is convex and $d_{p+1}(\cdot)$ is uniformly convex, the minimization problem \eqref{eq:prox} has a unique solution that we assume to be computable at reasonable cost. Let us first see how the exact solution of \eqref{eq:prox} satisfies \eqref{eq:A}. The first-order optimality conditions for \eqref{eq:prox} ensure that
\begin{align*}
   H \|T-\bar x\|^{p-1}B(\bar x-T)- \nabla f(T) \in\partial \psi(T).
\end{align*}
Thus, for $g=H \|T-\bar x\|^{p-1}B(\bar x-T)- \nabla f(T)$, the inequality in \eqref{eq:A} holds with any $\beta\in[0,1)$, i.e., $\prox(\bar x),g)\in \A(\bar x,\beta)$. Furthermore, since $\nabla f_{\bar x, H}^p(\bar x)=\nabla f(\bar x)$, we have $(\bar x,g)\not\in\A(\bar x,\beta)$ except if $\bar x = x^*$. In the next subsection, we show that an acceptable approximation of the operator \eqref{eq:prox} can be computed by applying one step of the $p$th-order tensor method (see \cite{nesterov2019implementable}) satisfying \eqref{eq:A}, while a lower-order method will be presented in Section \ref{sec:nepgm}. Let us highlight here that we are not able to find an inexact solution in the sense of \eqref{eq:A} for all points $\bar x$ in a neighbourhood of the solution $x^*$; however its exact solution always satisfies this inequality. We study this in the following example.

\vspace{-1mm}
\begin{exa}
    Let us consider the minimization of function $\func{f}{\R}{\R}$ given by $f(x)=x$ over the set $Q=\set{x\in\R ~:~x\geq 0}$, where $x^*=0$ is its unique solution. The indicator function of the set $Q$ is given by $\func{\psi}{\R}{\R}$ that is
    \begin{align*}
        \psi(x) = \delta_Q(x) = \left\{
        \begin{array}{ll}
            0       & ~~\mathrm{if}~ x\geq 0, \\
            +\infty & ~~\mathrm{if}~ x<0,
        \end{array}
        \right.
    \end{align*}
    where its subdifferential is given by
    \begin{align*}
        \partial \psi(x) = \left\{
        \begin{array}{ll}
            (-\infty,0]       & ~~\mathrm{if}~ x=0, \\
            \{0\} & ~~\mathrm{if}~ x>0,\\
            \emptyset & ~~\mathrm{if}~ x<0.
        \end{array}
        \right.
    \end{align*}
    Let us set $H=1, B=1, p=3, \bar x \neq 0$. Hence, $f_{\bar x, H}^3(x)=x+\tfrac{1}{4} |x-\bar x|^4$ that for $x\geq 0$ and $g\in\partial \psi(x)$ yield $\|\nabla f_{\bar x, H}^3(x)+g\|_*=|1+g+(x-\bar x)^3|,~ \|\nabla f(x)+g\|_*=|1+g|$. Therefore, for $\beta\in[0,1)$, the inequality $\|\nabla f_{\bar x, H}^3(T)+g\|_*\leq \beta\|\nabla f(T)+g\|_*$ leads to $|1+g+(T-\bar x)^3|\leq \beta|1+g|$, i.e.,
    \begin{align*}
        \bar x-\sqrt[3]{1+g+\beta|1+g|} \leq T \leq \bar x-\sqrt[3]{1+g-\beta|1+g|},\quad T\geq 0.
    \end{align*}
    It is clear that there is no $T>0$ (i.e., $g=0$) such that the right-hand-side inequality holds if we have $\bar x <\sqrt[3]{1-\beta}$ (see Subfigure~(a) of Figure~\ref{fig:fig1}). In this case, only the exact solution $T=0$ of the auxiliary problem satisfies the inequality \eqref{eq:A}.
    Indeed, $(T,g)\in\A(\bar x,\beta)$ if we have
    \begin{align*}
        T \in \left\{
        \begin{array}{ll}
            \left[\bar x-\sqrt[3]{1+g+\beta|1+g|},\bar x-\sqrt[3]{1+g-\beta|1+g|}\right] & ~~\mathrm{if}~ \bar x-\sqrt[3]{1+g+\beta|1+g|}\geq 0, \vspace{1mm}\\
            \left[0,\bar x-\sqrt[3]{1+g-\beta|1+g|}\right]     & ~~\mathrm{if}~ \bar x-\sqrt[3]{1+g+\beta|1+g|}<0,
        \end{array}
        \right.
    \end{align*}
    which we illustrate in Subfigure~(b) of Figure~\ref{fig:fig1} for some special choices of $\beta$ and $\bar x$.
    
    \begin{figure}[ht]
        \subfloat[\vspace{-2mm} No acceptable solutions for $\bar x=0.6$.]{\includegraphics[width= 7.4cm]{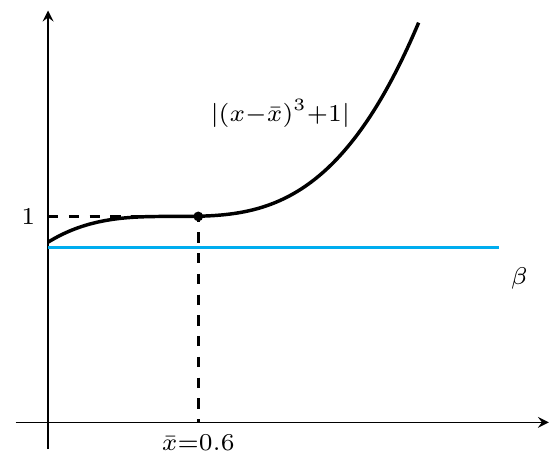}}\qquad
        \subfloat[\vspace{-2mm} Acceptable solutions for $\bar x=1.4$.]{\includegraphics[width= 7.4cm]{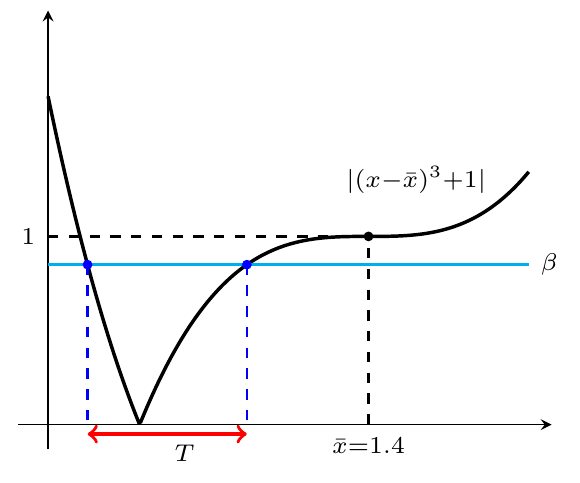}}\\
        \vspace{-1mm}
        \caption{
        Subfigure (a) shows that for $\bar x <\sqrt[3]{1-\beta}$, only the exact solution of the auxiliary problem satisfies \eqref{eq:A}, and Subfigure (b) illustrates the set of solutions for $\bar x=1.4$ and $\beta=0.85$ satisfying $\bar x\geq \sqrt[3]{1+\beta}$.\label{fig:fig1}}
    \end{figure}
\end{exa}

We first present the following lemma, which is a direct consequence of the definition of acceptable solutions \eqref{eq:A}.

\begin{lem}[properties of acceptable solutions]\label{lem:solProp}
    Let $(T,g)\in  \A(\bar x,\beta)$ for some $g\in\partial\psi(T)$. Then, we have
    \begin{equation}\label{eq:solProp1}
        (1-\beta)\|\nabla f(T)+g\|_* \leq H \|T-\bar x\|^p \leq (1+\beta)\|\nabla f(T)+g\|_*,
    \end{equation}
    \begin{equation}\label{eq:solProp2}
        \innprod{\nabla f(T)+g}{\bar x-T} \geq \tfrac{H}{1+\beta} \|T-\bar x\|^{p+1}.
    \end{equation}
    If additionally $\beta\leq\tfrac{1}{p}$, then
    \begin{equation}\label{eq:solProp3}
        \innprod{\nabla f(T)+g}{\bar x-T} \geq \left(\tfrac{1-\beta}{H}\right)^{1/p} \|\nabla f(T)+g\|_*^{\tfrac{p+1}{p}}.
    \end{equation}
\end{lem}

\begin{proof}
    From \eqref{eq:A} and the reverse triangle inequality, we obtain
    \begin{align*}
        \big|H\|T-\bar x\|^p-\|\nabla f(T)+g\|_*\big|\leq \|\nabla f(T)+H\|T-\bar x\|^{p-1}B(T-\bar x)+g\|_* \leq \beta \|\nabla f(T)+g\|_*,
    \end{align*}
    i.e., the inequality \eqref{eq:solProp1} holds. Squaring both sides of the inequality in \eqref{eq:A}, we come to
    \begin{align*}
        \|\nabla f(T)+g\|_*^2+2H\|T-\bar x\|^{p-1} \innprod{\nabla f(T)+g}{B(T-\bar x)} &+H^2\|T-\bar x\|^{2p}
        \leq \beta^2 \|\nabla f(x)+g\|_*^2,
    \end{align*}
    leading to
    \begin{align}
         \innprod{\nabla f(T)+g}{B(\bar x-T)}
         &\geq \tfrac{1-\beta^2}{2H\|T-\bar x\|^{p-1}} \|\nabla f(T)+g\|_*^2+\tfrac{H}{2}\|T-\bar x\|^{p+1} \label{eq:solProp31}\\
         &\geq \tfrac{H(1-\beta^2)}{2(1+\beta)^2} \|T-\bar x\|^{p+1}+\tfrac{H}{2}\|T-\bar x\|^{p+1} = \tfrac{H}{(1+\beta)} \|T-\bar x\|^{p+1}, \nonumber
    \end{align}
    giving \eqref{eq:solProp2}. Let us consider the function $\func{\zeta}{\R_+}{\R}$ with $\zeta(r)=\tfrac{1-\beta^2}{2H r^{p-1}} \|\nabla f(T)+g\|_*^2+\tfrac{H}{2} r^{p+1}$, which is the right-hand side of the inequality \eqref{eq:solProp31} with $r=\|T-\bar x\|$. From the inequality \eqref{eq:solProp1}, we obtain $r\geq \widehat{r}= \left(\tfrac{1-\beta}{H}\|\nabla f(T)+g\|_*\right)^{1/p}$. Taking the derivative of $\zeta$ at $\widehat{r}$ and $\beta\leq \tfrac{1}{p}$, we get
    \begin{align*}
        \zeta'(\widehat{r}) = \left(\tfrac{(1-p)(1+\beta)}{2}+\tfrac{(p+1)(1-\beta)}{2}\right) \|\nabla f(T)+g\|_* = (1-\beta p)\|\nabla f(T)+g\|_*\geq 0.
    \end{align*}
    Together with \eqref{eq:solProp31}, this implies \eqref{eq:solProp3}.
\end{proof}
\subsection{Solving \eqref{eq:prox} with $p$th-order tensor methods} 
\label{sec:ptensor}

In this section, we assume that $f(\cdot)$ is $p$th-order differentiabe with $M_{p+1}(f)<+\infty$ and show that an acceptable solution satisfying the inequality \eqref{eq:A} can be obtained by applying one step of the tensor method given in \cite{nesterov2019implementable}. 

The Taylor expansion of the function $f(\cdot)$ at $x\in\E$ is denoted by
\begin{align*}
    \Omega_{x,p}(y) = f(x)+\sum_{k=1}^p\tfrac{1}{k!} D^k f(x)[y-x]^k,\quad y\in\E,
\end{align*}
and it holds that
\begin{equation}\label{eq:tensorGradIneq}
    \|\nabla f(y)-\nabla \Omega_{x,p}(y)\|_* \leq \tfrac{M_{p+1}(f)}{p!} \|y-x\|^p. 
\end{equation}
Let us define the \textit{augmented Taylor approximation} as
\begin{align*}
    \widehat{\Omega}_{x,p}(y) = \Omega_{x,p}(y)+\tfrac{M_{p+1}(f)}{(p+1)!} \|y-x\|^{p+1}.
\end{align*}
Note that if $M\geq M_{p+1}(f)$, then $F(y)\leq \widehat\Omega_{x,p}(y)+\psi(y)$, which is a uniform upper bound for $F(\cdot)$. In the case $M\geq p M_{p+1}(f)$, the function $\widehat\Omega_{x,p}(y)+\psi(y)$ is convex, as confirmed by \cite[Theorem 1]{nesterov2019implementable}, which implies that one will be able to minimize the problem \eqref{eq:prob} by the \textit{tensor step}, i.e.,
\begin{equation}\label{eq:tensor}
    T_{p,M}^{f,g}(x)=\displaystyle \argmin_{y\in\dom\psi} \widehat\Omega_{x,p}(y)+\psi(y).
\end{equation}
We next show that an approximate solution of \eqref{eq:tensor} can be employed as an acceptable solution of the proximal-point operator \eqref{eq:prox} by the inexact $p$th-order tensor method proposed in \cite{grapiglia2020inexact,nesterov2019implementable}.

\begin{lem}[acceptable solutions by the tensor method \eqref{eq:tensor}]\label{lem:tensorApprox}
Let $(1-\gamma)M>M_{p+1}(f)$ and the approximate solution $T$ of \eqref{eq:tensor} satisfies
\begin{equation}\label{eq:approxTensor}
    \|\nabla \widehat\Omega_{x,p}(T)+ g\|_* \leq \tfrac{\gamma}{1+\gamma}\|\nabla \Omega_{x,p}(T)+ g\|_*,
\end{equation}
for some $g\in\partial \psi(T)$ and $\gamma\in \left[0, \tfrac{\beta}{1+\beta}\right)$. Then, for point $T=T_{p,M}^{f,g}(x)$, it holds that
\begin{equation}\label{eq:tensorSolIneq}
    \|\nabla f(T)+\tfrac{M}{p!}\nabla d_{p+1}(T-x)+g\|_* \leq \tfrac{M_{p+1}(f)+\gamma M}{(1-\gamma)M-M_{p+1}(f)} \|\nabla f(T)+ g\|_*.
\end{equation}
\end{lem}

\begin{proof}
    It follows from \eqref{eq:approxTensor} that
    \begin{align*}
        \tfrac{\gamma}{1+\gamma}\|\nabla \Omega_{x,p}(T)+ g\|_* \geq \|\nabla \widehat\Omega_{x,p}(T)+ g\|_* \geq \|\nabla \Omega_{x,p}(T)+ g\|_*-\tfrac{M}{p!}\|T-x\|^p,
    \end{align*}
    which consequently implies
    \begin{align*}
        \|\nabla \Omega_{x,p}(T)+ g\|_* \leq (1+\gamma) \tfrac{M}{p!}\|T-x\|^p,
    \end{align*}
    for some $g\in\partial \psi(T)$. Together with \eqref{eq:tensorGradIneq} and \eqref{eq:approxTensor}, this yields
    \begin{align*}
        \tfrac{M_{p+1}(f)}{p!} \|T-x\|^p &\geq \|\nabla f(T)-\nabla \Omega_{x,p}(T)\|_* = \|\nabla f(T)-\nabla \widehat\Omega_{x,p}(T)+\tfrac{M}{p!}\nabla d_{p+1}(T-x)\|_*\\
        &= \|\nabla f(T)+\tfrac{M}{p!}\nabla d_{p+1}(T-x)+g-(\nabla \widehat\Omega_{x,p}(T)+g)\|_*\\
        &\geq \|\nabla f(T)+\tfrac{M}{p!}\nabla d_{p+1}(T-x)+g\|_*-\|\nabla \widehat\Omega_{x,p}(T)+g\|_*\\
        &\geq \|\nabla f(T)+\tfrac{M}{p!}\nabla d_{p+1}(T-x)+g\|_*-\tfrac{\gamma M}{p!}\|T-x\|^p\\
        &\geq \|\tfrac{M}{p!}\nabla d_{p+1}(T-x)\|_*-\|\nabla f(T)+g\|_*-\tfrac{\gamma M}{p!}\|T-x\|^p,
    \end{align*}
    implying $\tfrac{1}{p!} \|T-x\|^p\leq \tfrac{1}{(1-\gamma) M-M_{p+1}(f)} \|\nabla f(T)+g\|_*$. This and the inequality
    \begin{align*}
         \|\nabla f(T)+\tfrac{M}{p!}\nabla d_{p+1}(T-x)+g\|_* &\leq  \tfrac{M_{p+1}(f)+\gamma M}{p!} \|T-x\|^p,
    \end{align*}
obtained in the above chain, leads to the desired result \eqref{eq:tensorSolIneq}.
\end{proof}

We note that setting $M=\tfrac{1+\beta}{\beta(1-\gamma)-\gamma}M_{p+1}(f)$ and $H=\tfrac{M}{p!}$, the inequality \eqref{eq:tensorSolIneq} can be rewritten in the form
\begin{align*}
    \|\nabla f(T)+H\nabla d_{p+1}(T-x)+g\|_* \leq \beta \|\nabla f(T)+g\|_*,
\end{align*}
which implies $T\in\A(x,\beta)$. In order to illustrate the results of Lemma \ref{lem:tensorApprox}, we study the following one-dimensional example.

\begin{exa}\label{exa:tensor}
    Let us consider the minimization of the one-dimensional function $\func{F}{\R}{\R}$ given by $F(x)=x^4+|x|$, where $x^*=0$ is its unique solution. In the setting of the problem \eqref{eq:prob}, we have $f(x)=x^4$ and $\psi(x)=|x|$. Let us set $p=3$, i.e., we have $M_4(f)=24$ and
    \begin{equation*}
        \Omega_{x,3}(y)= x^4+4x^3(y-x)+6x^2(y-x)^2+4x(y-x)^3,\quad \widehat{\Omega}_{x,3,M}(y)=\Omega_{x,3}(y)+\tfrac{M}{24} (y-x)^4,
    \end{equation*}
    where $M=1.9 M_4(f)$. Thus,
    \begin{equation*}
        \Omega_{x,3}'(y)= 4x^3+12x^2(y-x)+12x(y-x)^2,\quad \widehat{\Omega}_{x,3,M}'(y)=\Omega_{x,3}'(y)+\tfrac{M}{6} (y-x)^3.
    \end{equation*}
    Setting $\gamma=\tfrac{8}{19}\in[0,\tfrac{9}{19})$ and $x=0.8$, we illustrate the feasible area of $|\widehat{\Omega}'_{x,3,M}(y)|\leq \tfrac{\gamma}{1+\gamma}|\Omega'_{x,3}(y)|$ and acceptable solutions in Subfigures (a) and (b) of Figure \ref{fig:tensor}, respectively. We note that with our choice of $\gamma$ and $M$, we have $(1-\gamma)M>M_4(f)$, which implies that all assumptions of Lemma~\ref{lem:tensorApprox} are valid. 
    
    \begin{figure}[H]
        \subfloat[Points satisfying $|\widehat{\Omega}'_{x,3,M}(y)|\leq \tfrac{\gamma}{1+\gamma}|\Omega'_{x,3}(y)|$.]{\includegraphics[width= 7.4cm]{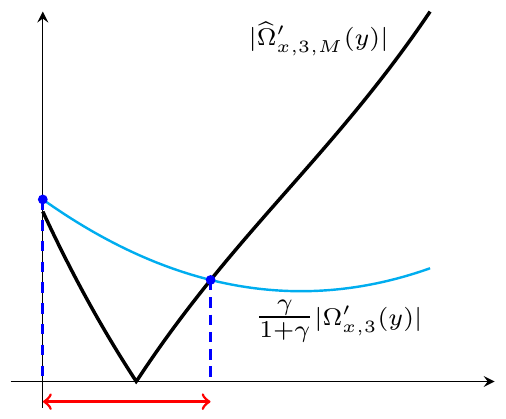}}\qquad
        \subfloat[Acceptable solutions for $\beta=0.9$ and $ x=0.8$.]{\includegraphics[width= 7.4cm]{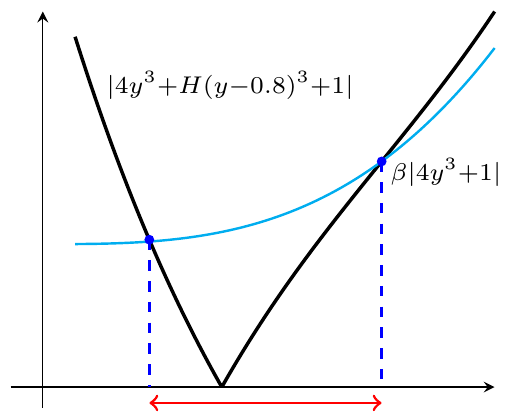}}
        \caption{Subfigure (a) stands for the set of points $y$ satisfying the inequality $|\widehat{\Omega}'_{x,3,M}(y)|\leq \tfrac{\gamma}{1+\gamma}|\Omega'_{x,3}(y)|$ with $x=0.8$, $\gamma=8/19$, and $\beta=0.9$, and Subfigure (b) illustrates the set of acceptable solutions for $x=0.8$ and $\beta =0.9$. \label{fig:tensor}}
    \end{figure}
\end{exa}

In Section \ref{sec:biLevel}, we further extend our discussion concerning the computation of an acceptable solution $\A(\bar x,\beta)$ for the $p$th-order proximal-point problem \eqref{eq:prox} by the lower-level methods.

\subsection{Inexact high-order proximal-point method}\label{sec:ihopp}
In this section, we introduce our inexact high-order proximal-point method for the composite minimization \eqref{eq:prob} and verify its rate of convergence. 

We now consider our first inexact high-order proximal-point scheme that generates a sequence of iterations satisfying 
\begin{equation}\label{eq:ihopp}
             (T_k,g) \in \A(x_k,\beta),
        \end{equation}
which we summarize in Algorithm~\ref{alg:ihoppa}.

\vspace{3mm}
\RestyleAlgo{boxruled}
\begin{algorithm}[H]
\DontPrintSemicolon \KwIn{ $x_{0}\in\dom \psi$,~ $\beta\in [0,1/p]$,~ $H>0$,~$\varepsilon>0$,~ $k=0$;} 
\Begin{ 
    \While {$F(x_k)-F^*>\varepsilon$}{  
        Find $T_k$ and $g\in\partial\psi(x_{k+1})$ such that \eqref{eq:ihopp} holds and set $x_{k+1}=T_k$;
        $k = k+1$;
     } 
}
\caption{Inexact High-Order Proximal-Point Algorithm\label{alg:ihoppa}}
\end{algorithm}

\vspace{3mm}
In order to verify the the convergence rate of Algorithm \ref{alg:ihoppa}, we need the next lemma, which was proved in \cite[Lemma 11]{nesterov2019inexact}.

\begin{lem}\cite[Lemma 11]{nesterov2019inexact}
\label{lem:xiConv}
    Let $\seq{\xi_k}$ be a sequence of positive numbers satisfying
    \begin{equation}\label{eq:xikIneq}
        \xi_k-\xi_{k+1} \geq \xi_{k+1}^{1+\alpha},\quad k\geq 0,
    \end{equation}
    for $\alpha\in (0,1]$. Then, for $k\geq 0$, the following holds
    \begin{equation}\label{eq:xikIneqRes}
        \xi_k \leq \tfrac{\xi_0}{\left(1+\tfrac{\alpha k}{1+\alpha}\log(1+\xi_0^\alpha)\right)^{1/\alpha}} \leq \left((1+\tfrac{1}{\alpha})(1+\xi_0^\alpha).\tfrac{1}{k}\right)^{1/\alpha}.
    \end{equation}
\end{lem}

Let us investigate the rate of convergence of Algorithm \ref{alg:ihoppa}. Let us first define the \textit{radius of the initial level set} of the function $\psi$ in \eqref{eq:prob} as $D_0= \max_{x\in\dom\psi} \set{\|x-x^*\| ~:~F(x)\leq F(x_0)}<+\infty$. 

\begin{thm}[convergence rate of Algorithm \ref{alg:ihoppa}]
\label{thm:ihoppConv}
    Let the sequence $\seq{x_k}$ be generated by the inexact high-order proximal-point method \eqref{eq:ihopp} with $\beta\in [0,1/p]$. Then, for $k\geq 0$, we have
    \begin{equation}\label{eq:ihoppConv}
        F(x_k)-F^* \leq \tfrac{1}{2} \left(\tfrac{1}{1-\beta} H D_0^{p+1}+F(x_0)-F^*\right) \left(\tfrac{2p+2}{k}\right)^p.
    \end{equation}
\end{thm}

\begin{proof}
    From the convexity of $\psi(\cdot)$ and \eqref{eq:solProp3}, we obtain 
    \begin{equation*}
        F(x_k)-F(x_{k+1}) \geq \innprod{\nabla f(x_{k+1})+g}{x_{k+1}-x_k}\\
        \geq \left(\tfrac{1-\beta}{H}\right)^{1/p} \|\nabla f(x_{k+1})+g\|_*^{\tfrac{p+1}{p}},
    \end{equation*}
    with $g\in\partial\psi(x_{k+1})$ and $(x_{k+1},g) \in \A(x_k,\beta)$. By Cauchy-Schwartz inequalitiy, we get
    \begin{align*}
        F(x_{k+1})-F^* &\leq \innprod{\nabla f(x_{k+1})+g}{x_{k+1}-x^*}\leq \|\nabla f(x_{k+1})+g\|_*\|x_{k+1}-x^*\|\\
        &\leq D_0 \|\nabla f(x_{k+1})+g\|_*.
    \end{align*}
    It follow from the last two inequalities, that
    \begin{align*}
        F(x_k)-F(x_{k+1}) \geq \left(\tfrac{1-\beta}{H D_0^{p+1}}\right)^{1/p} \left(F(x_{k+1})-F^*\right)^{\tfrac{p+1}{p}}.
    \end{align*}
    Setting $\xi_k=\tfrac{1-\beta}{H D_0^{p+1}}(F(x_k)-F^*)$ and $\alpha=1/p$, we see that the condition \eqref{eq:xikIneq} is satisfied for all $k\geq 0$. Therefore, from Lemma \ref{lem:xiConv}, we have
    \begin{align*}
        \xi_k \leq \left(\left(1+\tfrac{1}{\alpha}\right)(1+\xi_0^\alpha).\tfrac{1}{k}\right)^{\tfrac{1}{\alpha}} \leq \left(1+\tfrac{1}{\alpha}\right)^{\tfrac{1}{\alpha}} 2^{\tfrac{1-\alpha}{\alpha}}(1+\xi_0) \left(\tfrac{1}{k}\right)^{\tfrac{1}{\alpha}},
    \end{align*}
    giving \eqref{eq:ihoppConv}.
\end{proof}

\subsection{Accelerated inexact high-order proximal-point method}\label{sec:aihopp}
In this section, we accelerate the scheme \eqref{eq:ihopp} by applying a variant of the \textit{standard estimating sequences technique}, which has been used as an standard tools for accelerating first- and second-order methods; see, e.g., \cite{ahookhosh2019accelerated,baes2009estimate,nesterov2005smooth,nesterov2008accelerating,nesterov2013gradient,nesterov2015universal,nesterov2018lectures}. 

Let $\seq{A_k}$ be a sequence of positive numbers generated by $A_{k+1}=A_k+a_{k+1}$ for $a_k>0$. The idea of the estimating sequences techniques is to generate a sequence of estimating functions $\seq{\Psi_k(x)}$ of $F(\cdot)$ in such a way that, at each iteration $k\geq 0$, the inequality
\begin{equation}\label{eq:estSeqIneq}
    A_k F(x_k)\leq \Psi_k^*\equiv \min_{x\in\dom\psi} \Psi_k(x), \quad k\geq 0.
\end{equation}
Let us set $c_p=\left(\tfrac{1-\beta}{H}\right)^{1/p}$. Following  \cite{nesterov2020inexact,nesterov2020superfast}, we set
\begin{equation}\label{eq:Ak}
    A_{k}= \left(\tfrac{c_p}{2}\right)^p \left(\tfrac{k}{p+1}\right)^{p+1}, \quad a_{k+1}=A_{k+1}-A_k, \quad k\geq 0.
\end{equation}
For $x_0, y_k\in\E$ and $(T_k,g)\in\A(y_k,\beta)$, let us define the \textit{estimating sequence} 
\begin{equation}\label{eq:estSeq}
    \Psi_{k+1}(x)=\left\{
    \begin{array}{ll}
        d_{p+1}(x-x_0) &~~\mathrm{if}~ k=0,\vspace{2mm}\\
        \Psi_k(x)+a_{k+1}[\ell_{T_k}(x)+\psi(x)] &~~\mathrm{if}~ k\geq 1.  
    \end{array}
    \right.
\end{equation}

\begin{lem}\label{lem:estSeqIneq}
    Let the sequence $\seq{\Psi_k(x)}$ be generated by \eqref{eq:estSeq} and $\upsilon_k=\argmin_{x\in\E} \Psi_k(x)$. Then, we have
    \begin{equation}\label{eq:psikPsiIneq}
        A_k F(x)+ d_{p+1}(x-x_0)\geq \Psi_k(x) \geq \Psi_k^*+\tfrac{1}{p+1}\left(\tfrac{1}{2}\right)^{p-1} \|x-\upsilon_k\|^{p+1}, \quad \forall x\in\dom\psi,~ k\geq 0.
    \end{equation}
\end{lem}

\begin{proof}
    The proof is given by induction on $k$. For $k=0$, $\Psi_0=d_{p+1}(x-x_0)$ and so \eqref{eq:psikPsiIneq} holds. We now assume that \eqref{eq:psikPsiIneq} holds for $k$ and show it for $k+1$. Then, it follows from \eqref{eq:estSeq} and the subgradient inequality that
    \begin{align*}
        \Psi_{k+1}(x)&=\Psi_k(x)+a_{k+1}[\ell_{x_{k+1}}(x)+\psi(x)]\\
        &\leq A_k F(x)+ d_{p+1}(x-x_0)+a_{k+1}[\ell_{x_{k+1}}(x)+\psi(x)]\\
        &\leq A_k F(x)+ d_{p+1}(x-x_0)+a_{k+1}F(x),
    \end{align*}
    leading to \eqref{eq:psikPsiIneq} for $k+1$. The right hand side inequality in \eqref{eq:psikPsiIneq} is a direct consequence of the definition of $\Psi_k(\cdot)$ and \eqref{eq:dp1LowerBound}. 
\end{proof}

We next present an accelerated version of the scheme \eqref{eq:A}.

\RestyleAlgo{boxruled}
\begin{algorithm}[ht!]
\DontPrintSemicolon \KwIn{$x_{0}\in\dom \psi$,~ $\beta\in [0,1/p]$,~ $H>0$,~ $y_0=\upsilon_0=x_0$,~ $\Psi_0=d_{p+1}(x-x_0)$,~$\varepsilon>0,~$ $k=0$;} 
\Begin{ 
    \While {$F(x_k)-F^*>\varepsilon$}{  
        Compute $\upsilon_{k}=\argmin_{x\in\E} \Psi_{k}(x)$ and compute $A_{k+1}$ and $a_{k+1}$ by \eqref{eq:Ak};\;
        Set $y_k= \tfrac{A_k}{A_{k+1}}x_k+\tfrac{a_{k+1}}{A_{k+1}}\upsilon_k$ and compute $(T_k,g) \in \A(y_k,\beta)$;\;
        Find $x_{k+1}$ such that $F(x_{k+1})\leq F(T_k)$;\;
        Update $\Psi_{k+1}(x)$ by \eqref{eq:estSeq} and set $k=k+1$;\;
     } 
}
\caption{Accelerated Inexact High-Order Proximal-Point Algorithm\label{alg:aihoppa}}
\end{algorithm}

In the subsequent result, we investigate the convergence rate of the sequence generated by the accelerated inexact high-order proximal-point method (Algorithm \ref{alg:aihoppa}).

\begin{thm}[convergence rate of Algorithm \ref{alg:aihoppa}]
\label{thm:aihppConv}
     Let the sequence $\seq{x_k}$ be generated by Algorithm \ref{alg:aihoppa} with $\beta\in [0,1/p]$. Then, the following statements holds:
     \begin{itemize}
         \item[(i)] \label{thm:aihppConv1}
         for all $k\geq 0$, the inequality \eqref{eq:estSeqIneq} holds;
         \item[(ii)] \label{thm:aihppConv2}
         for all $k\geq 0$, 
             \begin{equation}\label{eq:aihoppConv}
        F(x_k)-F^* \leq \tfrac{H}{2(1-\beta)} d_{p+1}(x_0-x^*) \left(\tfrac{2p+2}{k}\right)^{p+1};
    \end{equation} 
         \item[(iii)] \label{thm:aihppConv3}
         $d_{p+1}(\upsilon_k-x^*) \leq 2^{p-1} d_{p+1}(x_0-x^*)$,  for all $k\geq 0$.
     \end{itemize}
\end{thm}

\begin{proof}
    We first show by induction that \eqref{eq:estSeqIneq} holds. Since $A_0=0$ and $\Psi_0=d_{p+1}(x-x_0)$, it clearly holds for $k=0$. We now assume that inequality \eqref{eq:estSeqIneq} holds  for $k \geq 0$, and prove it for $k+1$. From \eqref{eq:psikPsiIneq}, the induction assumption $\Psi_k^*\geq A_kF(x_k)$, and the subgradient inequality, we obtain
    \begin{align*}
        \Psi_{k+1}^* &= \min_{x\in\dom\psi} \set{\Psi_k(x)+a_{k+1}[\ell_{T_k}(x)+\psi(x)]}\\
        &\geq  \min_{x\in\dom\psi} \set{\Psi_k^*+\sigma_p \|x-\upsilon_k\|^{p+1}+a_{k+1}[\ell_{T_k}(x)+\psi(x)]}\\
        &\geq  \min_{x\in\dom\psi} \set{A_k F(x_k)+a_{k+1}[\ell_{T_k}(x)+\psi(x)]+\sigma_p \|x-\upsilon_k\|^{p+1}}\\
        &\geq  \min_{x\in\dom\psi} \set{A_k F(T_k)+a_{k+1}[f(T_k)+\innprod{\nabla f(T_k)+g}{x-T_k}+\psi(T_k)]+\sigma_p \|x-\upsilon_k\|^{p+1}}\\
        & =  \min_{x\in\dom\psi} \set{A_{k+1}F(T_k)+\innprod{\nabla f(T_k)+g}{a_{k+1}(x-T_k)+A_k(x_k-T_k)}+\sigma_p \|x-\upsilon_k\|^{p+1}}\\
        &\geq  \min_{x\in\dom\psi} \set{A_{k+1}F(T_k)+\innprod{\nabla f(T_k)+g}{a_{k+1}(x-\upsilon_k)+A_{k+1}(y_k-T_k)}+ \sigma_p\|x-\upsilon_k\|^{p+1}},
    \end{align*}
    with $\sigma_p=\tfrac{1}{p+1}\left(\tfrac{1}{2}\right)^{p-1}$. For all $x\in\dom\psi$, we have
    \begin{align*}
        a_{k+1}\innprod{\nabla f(T_k)+g}{x-\upsilon_k}+\tfrac{1}{p+1}\left(\tfrac{1}{2}\right)^{p-1} \|x-\upsilon_k\|^{p+1} \geq -\tfrac{p}{p+1} 2^{\tfrac{p-1}{p}} \left(a_{k+1}\|\nabla f(T_k)+g\|_*\right)^{\tfrac{p+1}{p}}.
    \end{align*}
    It follows from \eqref{eq:solProp3} and $T_k\in\A(y_k,\beta)$ that
    \begin{equation*}
        \innprod{\nabla f(T_k)+g}{y_k-T_k} \geq c_p \|\nabla f(T_k)+g\|_*^{\tfrac{p+1}{p}}.
    \end{equation*}
    Combining the last three inequalities yields
    \begin{equation}\label{eq:psiStarIneq}
        \begin{split}
        \Psi_{k+1}^* &\geq A_{k+1}F(T_k)+c_p A_{k+1} \|\nabla f(T_k)+g\|_*^{\tfrac{p+1}{p}}-\tfrac{p}{p+1} 2^{\tfrac{p-1}{p}} \left(a_{k+1}\|\nabla f(T_k)+g\|_*\right)^{\tfrac{p+1}{p}}\\
        &= A_{k+1}F(T_k)+\left(c_p A_{k+1}-\tfrac{p}{p+1} 2^{\tfrac{p-1}{p}} a_{k+1}^{\tfrac{p+1}{p}}\right) \|\nabla f(T_k)+g\|_*^{\tfrac{p+1}{p}}\\
        &\geq A_{k+1}F(T_k)+\left(c_p A_{k+1}-2a_{k+1}^{\tfrac{p+1}{p}}\right) \|\nabla f(T_k)+g\|_*^{\tfrac{p+1}{p}}.
        \end{split}
    \end{equation}
    On the other hand, from \eqref{eq:Ak}, it can be deduced
    \begin{align*}
        \frac{(A_{k+1}-A_k)^{\tfrac{p+1}{p}}}{A_{k+1}}=\frac{c_p}{2}\frac{\left(\left(\tfrac{k+1}{p+1}\right)^{p+1}-\left(\tfrac{k}{p+1}\right)^{p+1}\right)^{\tfrac{p+1}{p}}}{\left(\tfrac{k+1}{p+1}\right)^{p+1}}=\frac{c_p}{2} \left(\tfrac{k+1}{p+1}-\tfrac{k}{p+1}\left(1-\tfrac{1}{k+1}\right)^p\right)^{\tfrac{p+1}{p}}\leq \frac{c_p}{2},
    \end{align*}
    leading to
    \begin{equation*}
        a_k^{\tfrac{p+1}{p}}\leq \tfrac{c_p}{2} A_{k+1}, \quad k\geq 0.
    \end{equation*}
    Together with \eqref{eq:psiStarIneq} and $f(T_k)\geq F(x_{k+1})$, this ensures $\Psi_{k+1}^*\geq A_{k+1}F(x_{k+1})$, i.e., the assertion~(i) holds. 
    
    Invoking the inequalities \eqref{eq:estSeqIneq} and \eqref{eq:psikPsiIneq}, we come to
    \begin{align*}
        F(x_k)-F^*\leq \tfrac{1}{A_k} d_{p+1}(x_0-x^*) = \left(\tfrac{2}{c_p}\right)^p \left(\tfrac{p+1}{k}\right)^{p+1} d_{p+1}(x_0-x^*),
    \end{align*}
    giving the inequality \eqref{eq:aihoppConv}.
    
    It follows from \eqref{eq:estSeqIneq}, \eqref{eq:psikPsiIneq}, $F(x_k)-F^*\geq 0$, and $x=x^*$ that
    \begin{align*}
        d_{p+1}(x_0-x^*) \geq -A_k F^*+\Psi_k^*+\left(\tfrac{1}{2}\right)^{p-1} d_{p+1}(\upsilon_k-x^*)\geq (F(x_k)-F^*)+\left(\tfrac{1}{2}\right)^{p-1}d_{p+1}(\upsilon_k-x^*),
    \end{align*}
    which leads to the assertion (iii).
\end{proof}

\section{Bi-level optimization framework}\label{sec:biLevel}
As we have seen in  the previous sections, solving the  convex composite problem \eqref{eq:prob} by an inexact high-order proximal-point method involves two steps: (i) choosing a $p$th-order proximal-point method as an upper-level scheme; (ii) choosing a lower-level method for computing a point $T\in\A(\bar x, \beta)$. This gives us two degrees of freedom in the strategy of finding a solution to the problem \eqref{eq:prob}. This is why we call this framework  \textit{Bi-level OPTimization} (BiOPT). At the upper level, we do not need to impose any assumption on the objective $F(\cdot)$ apart from its convexity. At the lower-level method, we need some additional assumption on this objective function. Moreover, in the BiOPT setting, the complexity of a scheme leans on the complexity of both upper- and lower-level methods.

On the basis of the results of Section \ref{sec:ptensor}, the auxiliary problem \eqref{eq:prox} can be solved by applying one step of the $p$th-order tensor method. This demands the computation of $i$th ($i=1,\ldots,p$) directional derivatives of function $f(\cdot)$ and the condition \eqref{eq:approxTensor}, which might not be practical in general. Therefore, we could try to apply a lower-order method to the auxiliary problem \eqref{eq:prox}, which leads to an efficient implementation of the BiOPT framework. This is the main motivation of the following sections.

\subsection{Non-Euclidean composite gradient method}\label{sec:nepgm}
Let us assume that $k$ is a fixed iteration of either Algorithm \ref{alg:ihoppa} or Algorithm \ref{alg:aihoppa}, and we need to compute an {\it acceptable solution} $z_k$ of \eqref{eq:prox} satisfying \eqref{eq:A}. To do so, we introduce a non-Euclidean composite gradient method and analyze the convergence properties of the sequence $\{z_i\}_{i\geq 0}$ generated by this scheme, which 
satisfies in the limit inequality \eqref{eq:A}. Our main tool for such developments is the \textit{relative smoothness condition} (see \cite{bauschke2016descent,lu2018relatively} for more details and examples).  

Notice that an acceptable solution of the auxiliary problem \eqref{eq:prox} requires that the function $\func{\varphi_k}{\E}{\R}$ given by
\begin{equation}\label{eq:psik}
    \varphi_k(z)=f_{y_k,H}^p(z)+\psi(z), \quad \forall k\geq 0,~ z\in\dom \psi
\end{equation}
be minimized approximately, delivering a point $y_k\in\dom\psi$, satisfying the inequality \eqref{eq:A} holds for given . 

Let us consider a simple example in which $\func{f}{\R}{\R}$ with $f \equiv 0$ and $y_k=0$. Then, the function $\func{f_{0,H}^2}{\R}{\R}$ defined as $f_{0,H}^2(z)=\tfrac{1}{3}|z|^3$
with $\nabla f_{0,H}^2(z)=|z|z$, which is not Lipschitz continuous. This shows that one cannot expect the Lipschitz smoothness of $f_{y_k,H}^p(\cdot)$ for $p\geq 2$. However, it can be shown that this function belongs to a wider class of functions called {\it relatively smooth}.

Let function $\func{\rho}{\E}{\R}$ be closed, convex, and differentiable. We call it a {\it scaling function}. Now, the non-symmetric \textit{Bregman} distance function $\func{\beta_\rho}{\E\times\E}{\R}$ 
with respec to $\rho$ is given by
\begin{equation}\label{eq:BregmanDist}
    \beta_\rho(x,y)=\rho(y)-\rho(x)-\innprod{\nabla \rho(x)}{y-x}.
\end{equation}
For $x, y, z\in\E$, it is easy to see (e.g., the proof of Lemma 3 in \cite{nesterov2019inexact}) that
\begin{equation}\label{eq:threepointIneq}
    \beta_\rho(x,z)-\beta_\rho(y,z)+\beta_\rho(y,x)=\innprod{\nabla \rho(y)-\nabla \rho(x)}{z-x}.
\end{equation}
For a convex function $\func{h}{\E}{\R}$, we say that $h(\cdot)$ is \textit{$L_h$-smooth relative to $\rho(\cdot)$} if there exists a constant $L_h>0$ such that $(L_h\rho-h)(\cdot)$ is convex, and we call it \textit{$\mu_h$-strongly convex relative to $\rho(\cdot)$} if there exists $\mu_h>0$ such that $(h-\mu\rho)(\cdot)$ is convex; cf. \cite{bauschke2016descent,lu2018relatively}. The constant $\kappa_h=\mu_h/L_h$ is called the \textit{condition number} of $h(\cdot)$ relative to the scaling function $\rho(\cdot)$.

In the following lemma, we characterize the latter two conditions.

\begin{lem}\cite[Proposition 1.1]{lu2018relatively}
\label{lem:sigmaL}%
	~The following assertions are equivalent:
	\begin{description}
	\item[(i)]
		\(h(\cdot)\) is \(L_h\)-smooth and \(\mu_h\)-strongly convex relative to the scaling function $\rho(\cdot)$;
	\item[(ii)]\label{thm:LipIneq}%
		\(
			\mu_h\beta_\rho(x,y)
		{}\leq{}
			h(x)-h(y)-\innprod{\nabla h(y)}{x-y}
		{}\leq{}
			L_h\beta_\rho(x,y);
		\)
	\item[(iii)]
		\(
			\mu_h\innprod{\nabla \rho(y)-\nabla \rho(x)}{y-x}
		{}\leq{}
			\innprod{\nabla h(y)-\nabla h(x)}{y-x}
		{}\leq{}
			L_h\innprod{\nabla \rho(y)-\nabla \rho(x)}{y-x};
		\)
	\item[(iv)]
		\(
			\mu_h\nabla^2\rho(x)
		{}\preceq{}
			\nabla^2 h(x)
		{}\preceq{}
			L_h\nabla^2\rho(x)
		\).
	\end{description}
\end{lem}

Let us introduce the following assumptions on the minimization problem \eqref{eq:psik}:
\begin{description}
    \item[{\bf(H1)}] $\rho(\cdot)$ is uniformly convex of degree $p+1$ with the modulus $\sigma>0$, i.e., $\beta_\rho(x,y)\geq \tfrac{\sigma}{p+1} \|y-x\|^{p+1}$;
    \item[{\bf(H2)}] there exist constants $\mu, L\geq 0$ such that the function $f_{y_k,H}^p(\cdot)$ is $L$-smooth and $\mu$-strongly convex relative to the scaling function $\rho(\cdot)$.
\end{description}
In this subsection, for the sake of generality, we assume the existence of the scaling function $\rho(\cdot)$ such that the conditions (H1)-(H2) hold; however, in Section~\ref{sec:supFastPth} we introduce a specific scaling function satisfying (H1)-(H2).

We are in position now to develop a non-Euclidean composite gradient scheme for minimizing \eqref{eq:psik} based on the assumptions (H1)-(H2). For given $y_k,z_i\in\dom\psi$ and $H, L>0$,  we introduce the non-Euclidean composite gradient scheme
\begin{equation}\label{eq:zk1}
    z_{i+1} = \argmin_{z\in\E}\set{\innprod{\nabla f_{y_k,H}^p(z_i)}{z-z_i}+\psi(z)+2L \br(z_i,z)},
\end{equation}
which is first-order method and the point $z_k^*$ denotes the optimal solution of \eqref{eq:zk1}. Note that the first-order optimality conditions for \eqref{eq:zk1} leads to the following variational principle
\begin{equation}\label{eq:varPrinciple}
    \innprod{\nabla f_{y_k,H}^p(z_i)+2L(\nabla\rho(z_{i+1})-\nabla\rho(z_i))}{z-z_{i+1}}+\psi(z)\geq \psi(z_{i+1}).
\end{equation}

For the sequence $\{z_i\}_{i\geq 0}$ generated by the scheme \eqref{eq:zk1}, we next show the monotonicity of the sequence $\{\varphi_k(z_i)\}_{i\geq 0}$.

\begin{lem}[non-Euclidean composite gradient inequalities]\label{lem:monotoneIneq} 
    Let $\{z_i\}_{i\geq 0}$ be generated by the scheme \eqref{eq:zk1}. Then, it holds that
    \begin{equation}\label{eq:monotonezk}
        \varphi_k(z_{i+1}) \leq \varphi_k(z_i)-L \br(z_i,z_{i+1}).
    \end{equation}
    Moreover, we have
    \begin{equation}\label{eq:betaRhoIneq1}
        \beta_{\rho}(z_{i+1},z) \leq \left(1-\tfrac{\kappa}{2}\right)^{i+1} \beta_{\rho}(z_0,z) +\tfrac{1}{2L}\left(\varphi_k(z)-\varphi_k(z_{i+1})\right).
    \end{equation}
\end{lem}

\begin{proof}
    Since $z_{i+1}$ is a solution of \eqref{eq:zk1}, it holds that 
    \begin{align*}
        \innprod{\nabla f_{y_k,H}^p(z_i)}{z_{i+1}-z_i}+\psi(z_{i+1})+2L \br(z_i,z_{i+1}) \leq \psi(z_i).
    \end{align*}
    Together with the $L$-smoothness of $f_{y_k,H}^p(\cdot)$ relative to $\rho(\cdot)$, this implies
    \begin{align*}
        f_{y_k,H}^p(z_{i+1}) &\leq f_{y_k,H}^p(z_i) + \innprod{\nabla f_{y_k,H}^p(z_i)}{z_{i+1}-z_i}+L \br(z_i,z_{i+1})\\
        &\leq f_{y_k,H}^p(z_i) + \psi(z_i)-\psi(z_{i+1})-L \br(z_i,z_{i+1}),
    \end{align*}
    giving \eqref{eq:monotonezk}.
    
    Setting $x=z_{i+1}$ and $y=z_i$ in the three point inequality \eqref{eq:threepointIneq} and applying the inequality \eqref{eq:varPrinciple}, it can be concluded that
    \begin{align*}
        \beta_{\rho}(z_{i+1},z) -\beta_{\rho}(z_i,z) &= \innprod{\nabla \rho(z_i)-\nabla \rho(z_{i+1})}{z-z_{i+1}}-\beta_{\rho}(z_i,z_{i+1})\\
        &\leq \tfrac{1}{2L}\left[\innprod{\nabla f_{y_k,H}^p(z_i)}{z-z_{i+1}}+\psi(z)-\psi(z_{i+1})\right]-\beta_{\rho}(z_i,z_{i+1})\\
        &= \tfrac{1}{2L}\left[f_{y_k,H}^p(z_i)+\innprod{\nabla f_{y_k,H}^p(z_i)}{z-z_i}+\psi(z)\right]\\
        &~~~- \tfrac{1}{2L}\left[f_{y_k,H}^p(z_i)+\innprod{\nabla f_{y_k,H}^p(z_i)}{z_{i+1}-z_i}+\psi(z_{i+1})\right]-\beta_{\rho}(z_i,z_{i+1})\\
        &\leq \tfrac{1}{2L}\left[f_{y_k,H}^p(z_i)+\innprod{\nabla f_{y_k,H}^p(z_i)}{z-z_i}+\psi(z)-\varphi_k(z_{i+1})\right]\\
        &=\tfrac{1}{2L}\left[\varphi_k(z)-\varphi_k(z_{i+1})-\mu\beta_{\rho}(z_i,z)\right].
    \end{align*}
    Accordingly, we get
    \begin{align*}
        \beta_{\rho}(z_{i+1},z) &\leq \left(1-\tfrac{\kappa}{2}\right) \beta_{\rho}(z_i,z) +\tfrac{1}{2L}\left(\varphi_k(z)-\varphi_k(z_{i+1})\right)\\
        &\leq \cdots \leq\left(1-\tfrac{\kappa}{2}\right)^{i+1} \beta_{\rho}(z_0,z) +\tfrac{1}{2L}\left(\varphi_k(z)-\varphi_k(z_{i+1})\right),
    \end{align*}
    justifying the inequality \eqref{eq:betaRhoIneq1}.
\end{proof}

In summary, we come to the following non-Euclidean composite gradient algorithm.

\vspace{2mm}
\RestyleAlgo{boxruled}
\begin{algorithm}[H]
\DontPrintSemicolon \KwIn{$z_{0}=y_k\in \dom \psi$,~ $\beta\in[0,1/p]$,~ $L>0$,~ $i=0$;} 
\Begin{ 
    \Repeat{$\|\nabla f_{y_k, H}^p(z_i)+g\|_*\leq \beta \|\nabla f(z_i)+g\|_*$}{  
        Compute $z_{i+1}$ by \eqref{eq:zk1};\;
        Set $g=L(\nabla \rho(z_i)-\nabla \rho(z_{i+1}))-\nabla f_{y_k,H}^p(z_i)\in\partial\psi(z_{i+1})$ and $i=i+1$;
     } 
     $i_k^* = i$;\;
}
\KwOut{$z_k=z_{i_k^*}$ and $g=L(\nabla \rho(z_{i_k^*-1})-\nabla \rho(z_{i_k^*}))-\nabla f_{y_k,H}^p(z_{i_k^*-1})\in\partial\psi(z_{i+1})$.}
\caption{ Non-Euclidean Composite Gradient Algorithm\label{alg:zk1}}
\end{algorithm}
\vspace{2mm}

We now assume that the auxiliary problem \eqref{eq:zk1} can be solved exactly. For the sequence $\{z_i\}_{i\geq 0}$ given by \eqref{eq:zk1}, we will stop the scheme as soon as $\|\nabla f_{y_k, H}^p(z_{i+1})+g\|_*\leq \beta \|\nabla f(z_{i+1})+g\|_*$ holds, and then we set $z_k=z_{i+1}$. In the remainder of this section, we show that the stopping criterion holds for $i$ large enough. 

Setting $z=z_0$ in the inequality \eqref{eq:betaRhoIneq1}, it follows the $(p+1)$-uniform convexity of $\rho(\cdot)$ that
\begin{align*}
        \|z_0-z_{i+1}\|^{p+1} &\leq \tfrac{p+1}{\sigma} \beta_{\rho}(z_{i+1},z_0) \leq  \tfrac{p+1}{2\sigma L}\left(\varphi_k(z_0)-\varphi_k(z_{i+1})\right) 
        \leq \tfrac{p+1}{2\sigma L}\left(\varphi_k(z_0)-\inf \varphi_k\right)\\
        &\leq  \tfrac{p+1}{2\sigma L}\left(F(y_k)-F^*\right)<+\infty,\quad \forall i\in\mathbb{N}.
    \end{align*}
Let us define the bounded convex set
\begin{equation}\label{eq:levelSetPsik}
    \mathcal{L}_k(z_0, \Delta_k) = \set{z\in\E ~:~ \|z_0-z\|\leq \Delta_k,~\varphi_k(z)\leq \varphi_k(z_0)}, \quad \Delta_k=\left(\tfrac{p+1}{2\sigma L}\left(F(y_k)-F^*\right)\right)^{1/(p+1)},
\end{equation}
i.e., $\{z_i\}_{i\geq 0} \subseteq \mathcal{L}_k(z_0, \Delta_k)$.

The next results shows that the sequence $\{\dist(0,\partial \varphi_k(z_i))\}_{i\geq 0}$ vanishes, for $\{z_i\}_{i\geq 0}$ generated by Algorithm \ref{alg:zk1}. For doing so, we also require that
\begin{description}
    \item[{\bf(H3)}] $\|\nabla^2\rho(\cdot)\|\leq \overline L$ on the set $\mathcal{L}_k(z_0, \Delta_k)$ with $\overline L>0$.
\end{description}


\begin{lem}[subsequential convergence]\label{lem:subseqConv}
    Let $\{z_i\}_{i\geq 0}$ be generated by Algorithm \ref{alg:zk1}. If (H1)--(H3) hold, then 
    \begin{equation}\label{eq:Gi1Ineq}
        \varphi_k(z_i)-\varphi_k(z_{i+1}) \geq C\|\mathcal{G}_{i+1}\|_*^{p+1}, \quad C=\tfrac{L\sigma}{(p+1)(L-\mu)^{p+1} \overline L^{p+1}},
    \end{equation}
    where 
    \begin{equation}\label{eq:Gi1}
        \mathcal{G}_{i+1} = \nabla f_{y_k,H}^p(z_{i+1})+g,\quad g=L(\nabla \rho(z_i)-\nabla \rho(z_{i+1}))-\nabla f_{y_k,H}^p(z_i).
    \end{equation}
    This consequently implies
    \begin{equation}\label{eq:limSubdiff}
        \lim_{i\to +\infty } \dist(0,\partial \varphi_k(z_{i+1})) = 0.
    \end{equation}
\end{lem}

\begin{proof}
    Writing the first-order optimality conditions for the minimization problem \eqref{eq:zk1}, there exists $g\in\partial \psi(z_{i+1})$ such that
    \begin{align*}
        \nabla f_{y_k,H}^p(z_i)+g+L(\nabla \rho(z_{i+1})-\nabla \rho(z_i))=0,
    \end{align*}
    leading to
    \begin{align*}
        g=L(\nabla \rho(z_i)-\nabla \rho(z_{i+1}))-\nabla f_{y_k,H}^p(z_i).
    \end{align*}
    From the convexity of $f_{y_k,H}^p(\cdot)$ and $\psi(\cdot)$, we obtain $\partial \varphi_k (\cdot) = \nabla f_{y_k,H}^p(\cdot)+\partial \psi(\cdot)$, i.e., 
    \begin{align*}
        \mathcal{G}_{i+1} = \nabla f_{y_k,H}^p(z_{i+1})+g\in \partial \varphi_k (z_{i+1}).
    \end{align*}
    From the Lipschitz continuity of $\nabla f_{y_k,H}^p(\cdot)$ and $\nabla \rho(\cdot)$ on the bounded set $\mathcal{L}_k(z_0, \Delta_k)$, we obtain
    \begin{align*}
        -\mathcal{G}_{i+1} &= L (\nabla \rho(z_{i+1})-\nabla \rho(z_i)) - (\nabla f_{y_k,H}^p(z_{i+1})-\nabla f_{y_k,H}^p(z_i))\\
        &= \int_0^1 [(L \nabla^2\rho-\nabla^2 f_{y_k,H}^p)(z_i+\tau(z_{i+1}-z_i))](z_{i+1}-z_i)d\tau.
    \end{align*}
    Together with (H2), (H4), and Lemma~\ref{lem:sigmaL}(iv), this leads to 
    \begin{align*}
        \|\mathcal{G}_{i+1}\|_* &\leq \|(L \nabla^2\rho-\nabla^2 f_{y_k,H}^p)(z)\| \|z_{i+1}-z_i\| \leq (L-\mu)\|\nabla^2\rho(z)\| \|z_{i+1}-z_i\|\\
        & \leq (L-\mu) \overline L \|z_{i+1}-z_i\| \leq (L-\mu) \overline L \left(\tfrac{p+1}{\sigma} \br(z_i,z_{i+1})\right)^{1/(p+1)},
    \end{align*}
    where the last inequality comes from the uniform convexity of $\rho(\cdot)$ of order $p+1$. Thus, it can be concluded from \eqref{eq:monotonezk} that
    \begin{align*}
       \varphi_k(z_i)- \varphi_k(z_{i+1}) \geq L \br(z_i,z_{i+1})\geq \tfrac{L\sigma}{(p+1)(L-\mu)^{p+1} \overline L^{p+1}} \|\mathcal{G}_{i+1}\|_*^{p+1},
    \end{align*}
    giving \eqref{eq:Gi1Ineq}. Thus, $C\sum \|\mathcal{G}_{i+1}\|_*^{p+1} \leq \varphi_k(y_k)-\inf \varphi_k\leq F(y_k)-F^*<+\infty$, i.e., $\lim_{i\to\infty}\|\mathcal{G}_{i+1}\|=0$. Together with the inequality $\dist(0,\partial \varphi_k(z_{i+1}))\leq \|\mathcal{G}_{i+1}\|$, this implies \eqref{eq:limSubdiff}.
\end{proof}

We now show the well-definedness and complexity of Algorithm \ref{alg:zk1} in the subsequent result. 

\begin{thm}[well-definedness of Algorithm \ref{alg:zk1}]
\label{thm:wellDefinedAlg3}
Let us assume that all conditions of Lemma \ref{lem:subseqConv} hold, let $\{z_i\}_{i\geq 0}$ be a sequence generated by Algorithm~\ref{alg:zk1}, and let
    \begin{equation}\label{eq:assumPsiPsi*}
        F(z_i)-F(x^*)\geq \varepsilon, \quad \forall  i\geq 0,
    \end{equation}
    where $x^*$ is a minimizer of $F$ and $\varepsilon>0$ is the accuracy parameter. Moreover, assume that there exists a constant $D>0$ such that $\|z_i-x^*\|\leq D$ for all $i\geq 0$.
    Then, for the subgradients 
    \begin{align*}
        \mathcal{G}_{i_k^*} = \nabla f_{y_k,H}^p(z_{i_k^*})+g\in \partial \varphi_k (z_{i_k^*}),\quad g=L(\nabla \rho(z_{i_k^*-1})-\nabla \rho(z_{i_k^*}))-\nabla f_{y_k,H}^p(z_{i_k^*-1}) \in\partial\psi(z_{i_k^*}),
    \end{align*}
    and $z_{i_k^*}\in\dom \psi$, the maximum number of iterations $i_k^*$ needed to guarantee the inequality
    \begin{equation}\label{eq:AIneqNEPA2}
        \|\mathcal{G}_{i_k^*}\|_* \leq \beta \|\nabla f(z_{i_k^*})+g\|_*
    \end{equation}
    satisfies 
    \begin{equation}\label{eq:lowerBoundi1}
        i_k^*\leq 1+ \tfrac{2(p+1)}{\kappa}\log\left(\tfrac{\tfrac{D}{\beta} \left(\tfrac{2L}{C}\beta_{\rho}(z_0,z_k^*)\right)^{1/(p+1)}}{ \varepsilon} \right),
    \end{equation}
    where $C$ is defined in \eqref{eq:Gi1Ineq} and $\varepsilon>0$ is the accuracy parameter.
\end{thm}

\begin{proof}
    Combining the subgradient and Cauchy-Schwartz inequalities with $\|z_i-x^*\|\leq D$, it can be deduced that
    \begin{equation}\label{eq:normPartialPsi1}
        \|\nabla f(z_i)+g\|_*\geq \tfrac{F(z_i)-F^*}{\|z_i-x^*\|}\geq \tfrac{F(z_i)-F^*}{D}\geq \tfrac{\varepsilon}{D},
    \end{equation}
    for any $g\in \partial \psi(z_i)$. 
    From \eqref{eq:Gi1Ineq}, there exists $C>0$ such that
    \begin{align*}
        \varphi_k(z_{i_k^*-1})-\varphi_k(z_k^*)&\geq  \varphi_k(z_{i_k^*-1})-\varphi_k(z_{i_k^*}) \geq C\|\mathcal{G}_{i_k^*}\|_*^{p+1},
    \end{align*}
    for $\mathcal{G}_{i_k^*}=\nabla f_{y_k,H}^p(z_{i_k^*})+g\in\partial \varphi_k(z_{i_k^*})$ and $g\in\partial\psi(z_{i_k^*})$. Together with \eqref{eq:betaRhoIneq1}, this implies
    \begin{align*}
        \|\mathcal{G}_{i_k^*}\|_* &\leq C^{-1/(p+1)} \left(\varphi_k(z_{i_k^*-1})-\varphi_k(z_k^*)\right)^{1/(p+1)}\\ &\leq \left(\tfrac{2L}{C}\right)^{1/(p+1)}\left( 
         \left(1-\tfrac{\kappa}{2}\right)^{i_k^*-1} \beta_{\rho}(z_0,z_k^*) -\beta_{\rho}(z_{i_k^*-1},z_k^*)\right)^{1/(p+1)}\\
          &\leq \left(\tfrac{2L}{C}\beta_{\rho}(z_0,z_k^*)\right)^{1/(p+1)}
         \left(1-\tfrac{\kappa}{2}\right)^{(i_k^*-1)/(p+1)}.
    \end{align*}
    Since $1-\tfrac{\kappa}{2}\in(0,1)$, for large enough $i_k^*$, we have $\tfrac{\beta\varepsilon}{D}\leq \left(\tfrac{2L}{C}\beta_{\rho}(z_0,z_k^*)\right)^{1/(p+1)}
         \left(1-\tfrac{\kappa}{2}\right)^{(i_k^*-1)/(p+1)}$, i.e., the bound \eqref{eq:lowerBoundi1} is valid by \eqref{eq:normPartialPsi1} with $i={i_k^*}$.
\end{proof}
\subsection{Bi-level high-order methods}\label{sec:supFastPth}
In the BiOPT framework, we here consider Algorithm~\ref{alg:aihoppa} using the $p$th-order proximal-point operator in the upper-level, and in the lower-level we solve the auxiliary problem by the high-order non-Euclidean composite gradient method described in Algorithm \ref{alg:zk1}. As such, our proposed algorithm only needs the $p$th-order oracle for even $p$ and the $(p-1)$th-order oracle for odd $p$, which attains the complexity of order $\mathcal{O}(\varepsilon^{-1/(p+1)})$.

In the remainder of this section, we set $p\geq 2$ and $q=\floor{p/2}$. Let us define the function $\func{\rho_{y_k,H}}{\E}{\R}$ given by
\begin{equation}\label{eq:rhoBarxPth}
    \rho_{y_k,H}(x)= \sum_{k=1}^{q} \tfrac{1}{(2k)!} D^{2k} f(y_k)[x-y_k]^{2k}+H d_{p+1}(x-y_k),
\end{equation}
which is uniformly convex with degree $p+1$ that is not a trivial result. For $p=3$, the function \eqref{eq:rhoBarxPth} is reduced to $\rho_k(z)=\frac{1}{2}\innprod{\nabla^2f(y_k)(z-y_k)}{z-y_k}+3M_4(f) d_4(z-y_k)$ given in \cite{nesterov2020inexact}. Owing to this foundation, we can show that the function $f_{y_k,H}^p(\cdot)$ is $L$-smooth and $\mu$-strongly convex relative to the scaling function $\rho_{y_k,H}(\cdot)$, which paws the way toward algorithmic developments. We begin next with showing the uniform convexity of $\rho_{y_k,H}(\cdot)$. To this end, we need the $p$th-order Taylor expansion of the function $f$ around $y\in\dom f$ given by
\begin{equation}\label{eq:taylorExpan}
    f(x) = \Omega_{y,p}(x)+\tfrac{1}{p!} \int_0^1 (1-\xi)^p D^{p+1} f(y+\xi(x-y))[x-y]^{p+1} d\xi, 
\end{equation}
for $x\in\dom f$ and $\Omega_{y,p}(x)=f(y)+\sum_{k=1}^p \tfrac{1}{k!} D^k f(y)[x-y]^k$. It is not hard to show that
\begin{equation}\label{eq:nabla2fIneq}
    \nabla^2f(x)\preceq\nabla^2\Omega_{y,p}(x)+\tfrac{M_{p+1}(f)}{(p-1)!} \|x-y\|^{p-1} B,
\end{equation}
see \cite[Theorem 1]{nesterov2019implementable}.

\begin{thm}[uniform convexity and smoothness of $\rho_{y_k,H}(\cdot)$]
\label{thm:nabla2Ineq}
    For any $x-y_k\in\E$ and $\xi>1$, if $p\geq 2$ and $q=\floor{p/2}$, then 
    \begin{equation}\label{eq:nabla2Ineq}
       - \mathcal{M}_{y_k,p}(x)\preceq \sum_{k=1}^{q} \tfrac{1}{(2k-1)!} D^{2k+1} f(y_k)[x-y_k]^{2k-1} \preceq \mathcal{M}_{y_k,p}(x),
    \end{equation}
    where 
    \begin{align*}
    	\mathcal{M}_{y_k,p}(x)= \sum_{k=1}^{q} \tfrac{1}{(2k-2)!~\xi^{p-2k}} D^{2k} f(y_k)[x-y_k]^{2k-2}+\tfrac{\xi M_{p+1}(f)}{(p-1)!}\|x-y_k\|^{p-1}B.
    \end{align*}
    Moreover, the function $\rho_{y_k,H}(\cdot)$ given in \eqref{eq:rhoBarxPth} is uniformly convex with degree $p+1$, and the inequality $\|\nabla^2 \rho_{y_k,H}(\cdot)\|\leq \overline L$ holds for $\overline L=\tfrac{M_4(f)}{2} \Delta_k^2+M_2(f)+\left(\tfrac{M_{p+1}}{(p-1)!}+pH\right) \Delta_k^{p-1}$ on the set $\mathcal{L}_k(z_0, \Delta_k)$ if $M_{2}(f)<+\infty$, $M_{4}(f)<+\infty$, and $M_{p+1}(f)<+\infty$.
\end{thm}

\begin{proof}
    Let us fix arbitrary directions $u,h=x-y_k\in\E$. Setting $y=y_k$, it follows from \eqref{eq:nabla2fIneq} that
    \begin{align*}
        0&\leq \innprod{\nabla^2f(x)u}{u}\leq \innprod{\nabla^2\Omega_{y_k,p}(x)u}{u} +\tfrac{M_{p+1}(f)}{(p-1)!}\|h\|^{p-1} \|u\|^2\\
        &=\innprod{\sum_{k=2}^p \tfrac{1}{(k-2)!} D^k f(y_k)[h]^{k-2}u}{u} +\tfrac{M_{p+1}(f)}{(p-1)!}\|h\|^{p-1} \|u\|^2.
    \end{align*}
    Hence, replacing $h$ by $\xi h$ in the last inequality, dividing by $\xi^{p-2}$ for $\xi>1$, and splitting the sum into the odd and even terms, we come to 
    \begin{align*}
          -\innprod{\sum_{k=1}^{q} \tfrac{1}{(2k-2)!~\xi^{p-2k}} D^{2k} f(y_k)[h]^{2k-2}u}{u}&-\tfrac{\xi M_{p+1}(f)}{(p-1)!}\|h\|^{p-1} \|u\|^2\\ 
          &\leq \innprod{\sum_{k=1}^{q} \tfrac{1}{(2k-1)!~\xi^{p-1-2k}} D^{2k+1} f(y_k)[h]^{2k-1} u}{u}\\
          &\leq \innprod{\sum_{k=1}^{q} \tfrac{1}{(2k-1)!} D^{2k+1} f(y_k)[h]^{2k-1} u}{u},
    \end{align*}
    leading to the left hand side of \eqref{eq:nabla2Ineq}. Replacing $h$ by $-h$, it holds that
    \begin{align*}
        \innprod{\sum_{k=1}^{q} \tfrac{1}{(2k-1)!} D^{2k+1} f(y_k)[h]^{2k-1} u}{u} &\leq
          \innprod{\sum_{k=1}^{q} \tfrac{1}{(2k-2)!~\xi^{p-2k}} D^{2k} f(y_k)[h]^{2k-2}u}{u}\\
          &+\tfrac{\xi M_{p+1}(f)}{(p-1)!}\|h\|^{p-1} \|u\|^2,
    \end{align*}
    giving the right hand side of \eqref{eq:nabla2Ineq}.
    
    From the $p$th-order Taylor expansion of the function $f$ at $y_k$, \eqref{eq:nabla2fIneq}, \eqref{eq:nabla2Ineq}, and \eqref{eq:dp1Der}, we obtain
    \begin{equation}\label{eq:nabla2f}
        \begin{split}
        \nabla^2 f(x) & \preceq \sum_{k=2}^p \tfrac{1}{(k-2)!} D^k f(y_k)[h]^{k-2}+ \tfrac{1}{(p-1)!} M_{p+1}(f)\|h\|^{p-1} B\\
        & \preceq \sum_{k=1}^{q} \tfrac{1}{(2k-2)!} \left(1+\tfrac{1}{\xi^{p-2k}}\right) D^{2k} f(y_k)[h]^{2k-2}+ \tfrac{(1+\xi)}{(p-1)!} M_{p+1}(f)\|h\|^{p-1} B\\
        & \preceq \sum_{k=1}^{q} \tfrac{1}{(2k-2)!} \left(1+\tfrac{1}{\xi^{p-2k}}\right) D^{2k} f(y_k)[h]^{2k-2}+ \tfrac{(1+\xi)}{(p-1)!} M_{p+1}(f)\nabla^2 d_{p+1}(h).
        \end{split}
    \end{equation}
    Since $f(\cdot)$ is convex, this and \eqref{eq:xixi1Pth} imply
    \begin{align*}
        0\preceq \nabla^2 f(x) &\preceq \left(1+\tfrac{1}{\xi}\right) \left[\sum_{k=1}^{q} \tfrac{1}{(2k-2)!} D^{2k} f(y_k)[x-y_k]^{2k-2}+ H\nabla^2 d_{p+1}(x-y_k)\right]\\
        &=\left(1+\tfrac{1}{\xi}\right) \nabla^2 \rho_{y_k, H}(x),
    \end{align*}
    leading to the convexity of $\rho_{y_k, H}(\cdot)$. Moreover, its uniform convexity of the degree $p+1$ follows from that of $d_{p+1}(\cdot)$. 
    
    It follows from \eqref{eq:taylorExpan} that
\begin{align*}
    &\nabla^2 f(y_k+h) = \nabla^2 f(y_k)+\sum_{k=3}^p \tfrac{1}{(k-2)!} D^k f(y_k)[h]^{k-2}+r_{p+1}(h),\\
    &\nabla^2 f(y_k-h) = \nabla^2 f(y_k)+\sum_{k=3}^p (-1)^{k-2} \tfrac{1}{(k-2)!} D^k f(y_k)[h]^{k-2}+r_{p+1}(-h),
\end{align*}
where $\|r_{p+1}(\pm h)\|\leq \tfrac{M_{p+1}}{(p-1)!} \|h\|^{p-1}$. Summing up the latter identities, it holds that
\begin{equation}\label{eq:secondDer}
    \tfrac{1}{2}\left(\nabla^2 f(y_k+h)+\nabla^2 f(y_k-h)\right)-\tfrac{1}{2}(r_{p+1}(h)+r_{p+1}(h))= \sum_{k=1}^q \tfrac{1}{(2k-2)!} D^{2k} f(y_k)[h]^{2k-2}.
\end{equation}
Moreover, we have
\begin{align*}
    &\nabla^2 f(y_k+h) = \nabla^2 f(y_k)+ D^3 f(y_k)[h]+r_4(h),\\
    &\nabla^2 f(y_k-h) = \nabla^2 f(y_k)+D^3 f(y_k)[-h]+r_4(-h),
\end{align*}
leading to
\begin{align*}
    \|\tfrac{1}{2}(\nabla^2 f(y_k+h)+\nabla^2 f(y_k-h))-\nabla^2 f(y_k)\| \leq \tfrac{1}{2}(r_4(h)+r_4(-h))\leq \tfrac{M_4(f)}{2}\|h\|^2.
\end{align*}
    For $x\in\mathcal{L}_k(z_0, \Delta_k)$ and $h=x-y_k$, it follows that
    \begin{align*}
        \|\nabla^2 \rho_{y_k, H}(x)\|&\leq \left\|\sum_{k=1}^{q} \tfrac{1}{(2k-2)!} D^{2k} f(y_k)[h]^{2k-2} \right\|+ pH\|h\|^{p-1}\\
        &\leq \left\|\tfrac{1}{2}(\nabla^2 f(y_k+h)+\nabla^2 f(y_k-h))-\tfrac{1}{2}(r_{p+1}(h)+r_{p+1}(-h)) \right\|+ pH\|h\|^{p-1}\\
        &\leq \|\tfrac{1}{2}(\nabla^2 f(y_k+h)+\nabla^2 f(y_k-h))-\nabla^2 f(y_k)\| +\|\nabla^2 f(y_k)\|+\tfrac{M_{p+1}}{(p-1)!} \|h\|^{p-1}+ pH\|h\|^{p-1}\\
        &\leq \tfrac{M_4(f)}{2} \|h\|^2+M_2(f)+\tfrac{M_{p+1}}{(p-1)!} \|h\|^{p-1}+pH\|h\|^{p-1}\\
        &\leq \tfrac{M_4(f)}{2} \Delta_k^2+M_2(f)+\left(\tfrac{M_{p+1}}{(p-1)!}+pH\right) \Delta_k^{p-1},
    \end{align*}
    establishing the boundedness of $\|\nabla^2 \rho_{y_k, H}(\cdot)\|$ on the set $\mathcal{L}_k(z_0, \Delta_k)$.
\end{proof}

Theorem~\ref{thm:nabla2Ineq} is clearly implies that the assumptions (H1) and (H4) are satisfied for the scaling function $\rho_{y_k,H}(\cdot)$ \eqref{eq:rhoBarxPth}. In the subsequent result, we show that the assumption (H2) also holds for this function.

\begin{thm}[relative smoothness and strong convexity of $f_{y_k, H}^p(\cdot)$] \label{lem:relSmoothfbarPth}
    Let $H\geq M_{p+1}(f)$ and let $p\geq 2$ and $q=\floor{p/2}$. Then,  the function $\func{f_{y_k, H}^p}{\E}{\R}$ is $L$-smooth and $\mu$-strongly convex relative to $\rho_{y_k,H}(\cdot)$ defined in \eqref{eq:rhoBarxPth} with
    \begin{equation}\label{eq:parRelSmoothPth}
        \mu=1-\tfrac{1}{\xi}, \quad L=1+\tfrac{1}{\xi}, \quad \kappa = \tfrac{\xi-1}{\xi+1},
    \end{equation}
    where $\xi$ is the unique solution of the quadratic equation
    \begin{equation}\label{eq:xixi1Pth}
        \xi(1+\xi) = \tfrac{(p-1)!H}{M_{p+1}(f)}.
    \end{equation}
\end{thm}

\begin{proof}
    In light of $f_{y_k,H}^p(\cdot)=f(\cdot)+H d_{p+1}(\cdot-y_k)$, \eqref{eq:nabla2f}, and \eqref{eq:xixi1Pth}, we can write
    \begin{align*}
        \nabla^2 f_{y_k,H}^p(x) &\preceq \left(1+\tfrac{1}{\xi}\right) \sum_{k=1}^{q} \tfrac{1}{(2k-2)!}  D^{2k} f(y_k)[h]^{2k-2}\\
        &~~~+ \left[\tfrac{\xi(1+\xi)}{(p-1)!}M_{p+1}(f)+\tfrac{(1+\xi)}{(p-1)!} M_{p+1}(f)\right] \nabla^2 d_{p+1}(x-y_k)\\
        &= \left(1+\tfrac{1}{\xi}\right) \nabla^2 \rho_{y_k, H}(x).
    \end{align*}
    On the other hand, the $p$th-order Taylor expansion \eqref{eq:taylorExpan} and \eqref{eq:nabla2Ineq} yield
    \begin{align*}
        \nabla^2 f(x) &\succeq\sum_{k=2}^p \tfrac{1}{(k-2)!} D^k f(y_k)[x-y_k]^{k-2}- \tfrac{1}{(p-1)!} M_{p+1}(f)\|x-y_k\|^{p-1} B\\
        & \succeq \sum_{k=1}^{q} \tfrac{1}{(2k-2)!} \left(1-\tfrac{1}{\xi^{p-2k}}\right) D^{2k} f(y_k)[h]^{2k-2}- \tfrac{(1+\xi)}{(p-1)!} M_{p+1}(f)\|h\|^{p-1} B\\
        & \succeq \sum_{k=1}^{q} \tfrac{1}{(2k-2)!} \left(1-\tfrac{1}{\xi^{p-2k}}\right) D^{2k} f(y_k)[h]^{2k-2}- \tfrac{(1+\xi)}{(p-1)!} M_{p+1}(f)\nabla^2 d_{p+1}(h).
    \end{align*}
    We therefore have
    \begin{align*}
        \nabla^2 f_{y_k,H}(x) &\succeq \left(1-\tfrac{1}{\xi}\right) \sum_{k=1}^{q} \tfrac{1}{(2k-2)!} D^{2k} f(y_k)[h]^{2k-2}\\
        &~~~+ \left[\tfrac{\xi(1+\xi)}{(p-1)!}M_{p+1}(f)-\tfrac{(1+\xi)}{(p-1)!} M_{p+1}(f)\right] \nabla^2 d_{p+1}(h)\\
        &= \left(1-\tfrac{1}{\xi}\right) \nabla^2 \rho_{y_k, H}(x),
    \end{align*}
    giving our desired result.
 \end{proof}

Motivated by the equations \eqref{eq:xixi1Pth}, in the remainder of this section, we set
\begin{equation}\label{eq:parSuperfastAlg1}
    \xi=2,\quad H=\tfrac{6}{(p-1)!} M_{p+1}(f), \quad \mu=\tfrac{1}{2}, \quad L=\tfrac{3}{2}, \quad \kappa=\tfrac{1}{3}.
\end{equation}
Additionally, in view of \eqref{eq:Ak}, we consider
\begin{equation}\label{eq:parSuperfastAlg2}
    \beta=\tfrac{1}{p},\quad A_k=\tfrac{(p-1)(p-1)!}{3p2^{p+1} M_{p+1}(f)} \left(\tfrac{k}{p+1}\right)^{p+1}, \quad a_{k+1}=A_{k+1}-A_k, \quad \mathrm{for}~ k\geq 0.
\end{equation}
We now present our accelerated high-order method by combining all above facts with Algorithm \ref{alg:aihoppa} leading to the following algorithm.

\vspace{3mm}
\RestyleAlgo{boxruled}
\begin{algorithm}[H]
\DontPrintSemicolon \KwIn{$x_{0}\in\dom \psi$, $\beta\in [0,1/p]$, $H=\tfrac{6}{(p-1)!} M_{p+1}(f)$, $A_0=0$, $\Psi_0=d_{p+1}(x-x_0)$, $k=0$;} 
\Begin{ 
    \While {$F(x_k)-F^*>\varepsilon$}{ 
        Compute $\upsilon_{k}=\argmin_{x\in\E} \Psi_{k}(x)$ and compute $A_{k+1}$ and $a_{k+1}$ by \eqref{eq:Ak};\;
        Set $y_k= \tfrac{A_k}{A_{k+1}}x_k+\tfrac{a_{k+1}}{A_{k+1}}\upsilon_k$ and $z_0=y_k$ and consider the scaling function \eqref{eq:rhoBarxPth};\;
        Find $z_k=z_{i_k^*}$ of \eqref{eq:prox} and $g\in\partial \psi(z_k)$ by Algorithm \ref{alg:zk1} such that $(z_k^*, g) \in \A(y_k,\beta)$;\;
        Find $x_{k+1}$ such that $F(x_{k+1})\leq F(T_k)$;\;
        Update $\Psi_{k+1}(x)$ by \eqref{eq:estSeq} and set $k=k+1$;\;
     } 
}
\caption{(Bi-Level High-Order Algorithm)\label{alg:supFP1OA}}
\end{algorithm}
\vspace{3mm}

Now, let us have a look at the optimality conditions for the auxiliary  problem \eqref{eq:zk1} for our $p$th-order proximal-point operator given by
\begin{align*}
    \nabla f_{y_k, H}^p(z_i)+\partial \psi(z_{i+1})+2L\left(\nabla\rho_k(z_{i+1})-\nabla\rho_k(z_i)\right) \ni 0,
\end{align*}
which should be solved exactly in our setting. We next translate this inclusion for convex constrained problem \eqref{eq:conProb}.

\begin{exa}\label{exa:conProb}
We here revisit the convex constrained problem \eqref{eq:conProb} and its unconstrained version \eqref{eq:compProb} with $\psi(\cdot)=\delta_Q(\cdot)$. For given $z_i\in\E$, writing the first-order optimality conditions for this problem leads to 
\begin{equation}\label{eq:zi1NC}
    N_Q(z_{i+1}) \ni 2L\left(\nabla\rho_k(z_i)-\nabla\rho_k(z_{i+1})\right)-\nabla f_{y_k, H}^p(z_i),
\end{equation}
where $\partial \psi(z_{i+1})=N_Q(z_{i+1})$ and therefore the \emph{normal cone}
\begin{align*}
    N_Q(x)=\left\{
    \begin{array}{ll}
        \set{u\in\E ~:~ \innprod{u}{y-x}\leq 0,~ \forall y\in Q} &~\mathrm{if}~ x\in Q, \\
         \emptyset                                               &~\mathrm{if}~ x\notin Q
    \end{array}
    \right.
\end{align*}
plays a crucial role for finding a solution of the auxiliary problem \eqref{eq:zk1}. As an example, let us consider the Euclidean ball $Q=\set{x\in\R^n ~:~ \|x\|\leq \delta}$ for which we have
\begin{align*}
    N_Q(x)=\left\{
    \begin{array}{ll}
        \set{\alpha x ~:~ \alpha>0} &~\mathrm{if}~ \|x\|=\delta, \\
         \set{0}                    &~\mathrm{if}~ \|x\|<\delta.
    \end{array}
    \right.
\end{align*}
We now set $p=3$ and consider two cases: (i) $\|z_{i+1}\|<\delta$; (ii) $\|z_{i+1}\|=\delta$. In Case (i), we have
\begin{align*}
   2L\left(\nabla^2f(y_k)(z_{i+1}-z_i)+H\|z_{i+1}-y_k\|^2(z_{i+1}-y_k)-H\|z_i-y_k\|^2(z_i-y_k)\right)-\nabla f_{y_k, H}^3(z_i)=0,
\end{align*}
with $\nabla f_{y_k, H}^3(z_i)=\nabla f(z_i)+H\|z_i-y_k\|^2(z_i-y_k)$, i.e.,
\begin{align*}
   \left[\nabla^2f(y_k)+H\|z_{i+1}-y_k\|^2I\right](z_{i+1}-y_k)=b_i,
\end{align*}
for $b_i=\left[\nabla^2f(y_k)+H\|z_i-y_k\|^2I\right](z_i-y_k)+\tfrac{1}{2L}\nabla f_{y_k, H}^3(z_i)$. This consequently implies 
\begin{align*}
   z_{i+1}=y_k+\left[\nabla^2f(y_k)+Hr^2I\right]^{-1}b_i,
\end{align*}
where $r=\|z_{i+1}-y_k\|$ can be computed by solving the one-dimensional equation
\begin{align*}
    r=\left\|\left[\nabla^2f(y_k)+Hr^2I\right]^{-1}b_i\right\|.
\end{align*}
In Case (ii) ($\|z_{i+1}\|=\delta$), there exists $\alpha>0$ such that
\begin{align*}
   \left[\nabla^2f(y_k)+H\|z_{i+1}-y_k\|^2I\right](z_{i+1}-y_k)-b_i=\alpha z_{i+1},
\end{align*}
leading to
\begin{align*}
   z_{i+1}=y_k+\left[\nabla^2f(y_k)+(Hr^2-\alpha)I\right]^{-1}(b_i+\alpha y_k),
\end{align*}
where $r=\|z_{i+1}-y_k\|$ and $\alpha$ are obtained by solving the system
\begin{equation*}
    \left\{
    \begin{array}{l}
        r=\left\|\left[\nabla^2f(y_k)+(Hr^2-\alpha)I\right]^{-1}(b_i+\alpha y_k)\right\|,\\
        \delta = \left\|y_k+\left[\nabla^2f(y_k)+(Hr^2-\alpha)I\right]^{-1}(b_i+\alpha y_k)\right\|.
    \end{array}
    \right.
\end{equation*}
Finally, we come to the solution
\begin{align*}
    z_{i+1}=\left\{
    \begin{array}{ll}
        y_k+\left[\nabla^2f(y_k)+Hr^2I\right]^{-1}b_i &~\mathrm{if}~\left\|y_k+\left[\nabla^2f(y_k)+Hr^2I\right]^{-1}b_i\right\|<\delta,\\
        y_k+\left[\nabla^2f(y_k)+(Hr^2-\alpha)I\right]^{-1}(b_i+\alpha y_k) &~ \mathrm{otherwise},
    \end{array}
    \right.
\end{align*}
for the $r$ and $\alpha$ computed by solving the above-mentioned nonlinear systems.
\end{exa}

In order to upper bound the Bregman term $\beta_{\rho_k}(\cdot,\cdot)$, we next define the {\it norm-dominated scaling function} in the following, which will be needed in the remainder of this section.

\begin{defin}\cite[Definition 2]{nesterov2020superfast}
\label{def:normDominat}
    The scaling function $\rho(\cdot)$ is called \emph{norm-dominated} on the set $S\subseteq \E$ by some function $\func{\theta_S}{\R_+}{\R_+}$ if $\theta_S(\cdot)$ is convex with $\theta_S(0)=0$ such that
    \begin{equation}\label{eq:normDominat}
        \br(x,y) \leq \theta_S(\|x-y\|),
    \end{equation}
    for all $x\in S$ and $y\in\E$.
\end{defin}

From now on and for sake of simplicity, we denote $\rho_{y_k,H}(\cdot)$ by $\rho_k(\cdot)$. In order to show the norm-dominatedness of the scaling function $\rho_k(\cdot)$ \eqref{eq:rhoBarxPth} is norm-dominated, we first need the following technical lemma.

\begin{lem}[norm-dominatedness of the Euclidean ball]\label{lem:d4GradDomin}
    Let $p\geq 2$ and $q=\floor{p/2}$. Then, the function $d_{p+1}(\cdot)$ is norm-dominated on the Euclidean ball
    \begin{align}\label{eq:EuclidBallR}
        B_R = \set{x\in\E ~:~ \|x\|\leq R}
    \end{align}
    by the function
    \begin{equation}\label{eq:dp1GradDomin}
        \widehat{\theta}_R(\tau) = \left\{
        \begin{array}{ll}
           \alpha_1 a_1 \tau^{2q+2}+ \beta_1 d_1  &~~ \mathrm{if}~ p=2q+1,\vspace{2mm}\\
           \alpha_2 a_2 \tau^{2q+1}+ \beta_2 d_2  &~~ \mathrm{if}~ p=2q,
        \end{array}
        \right.
    \end{equation}
    where
    \begin{equation}\label{eq:alphaBetaRange}
        \begin{array}{l}
        1<\alpha_1 \leq 1+\tfrac{1}{2a_1}(b_1+1)c_1^{-\tfrac{q+1}{q}}, \quad 1+\tfrac{1}{2d_1} (b_1+1) c_1^{\tfrac{q+1}{q}}\leq \beta_1,\\
        1<\alpha_2 \leq 1+\tfrac{b_2+1}{2a_2}c_2^{-\tfrac{2q+1}{2(2q-1)}}, \quad 1+\tfrac{b_2+1}{2d_2} c_2^{\tfrac{2q+1}{2(2q-1)}}\leq \beta_2,
        \end{array}
    \end{equation}
    with 
    \begin{equation}\label{eq:abcd}
        \begin{array}{l}
         a_1=\tfrac{2^{2q}}{p+1}, \quad b_1=\tfrac{2^{3q+1}}{p+1} R^{q+1},\quad c_1=R^{p}
, \quad d_1=\tfrac{2^q}{p+1} R^{2q+2},\\
         a_2=2^{2q-1}, \quad b_2=\tfrac{2^{\tfrac{6q-1}{2}}}{p+1} R^{\tfrac{2q+1}{2}},\quad c_2=R^{p}, \quad d_2=\tfrac{2^{\tfrac{2q-1}{2}}}{p+1} R^{2q+1},
        \end{array}
    \end{equation}
    for $\tau\geq 0$. 
\end{lem}

\begin{proof}
    Let us first assume $p$ is an odd number, i.e., $p=2q+1$. For $x\in B_R$ and $y=x+h\in\E$, it follows from the inequality $\left(a^{1/t}+b^{1/t}\right)^t\leq 2^{t-1}(a+b)$ for $a,b\geq 0$ and $t\geq 1$ that
    \begin{align*}
        \beta_{d_{p+1}}(x,y) &= \tfrac{1}{p+1} \|y\|^{p+1}-\tfrac{1}{p+1} \|x\|^{p+1}- \|x\|^{p-1}\innprod{Bx}{y-x}\\
        &= \tfrac{1}{p+1} \left[\|x\|^2+\left(2\innprod{Bx}{h}+\|h\|^2\right)\right]^{q+1}-\tfrac{1}{p+1} \|x\|^{p+1}- \|x\|^{p-1}\innprod{Bx}{h}\\
        &\leq \tfrac{2^q}{p+1} \left[\|x\|^{2q+2}+\left(2\innprod{Bx}{h}+\|h\|^2\right)^{q+1}\right]-\tfrac{1}{p+1} \|x\|^{p+1}- \|x\|^{p-1}\innprod{Bx}{h}\\
        &\leq \tfrac{2^q}{p+1} \left[\|x\|^{2q+2}+2^q\left(2^{q+1}\innprod{Bx}{h}^{q+1}+\|h\|^{2q+2}\right)\right]-\tfrac{1}{p+1} \|x\|^{p+1}- \|x\|^{p-1}\innprod{Bx}{h}\\
        &\leq \tfrac{2^q}{p+1} \|x\|^{2q+2}+\tfrac{2^{3q+1}}{p+1} \|x\|^{q+1} \|h\|^{q+1}+\tfrac{2^{2q}}{p+1}\|h\|^{2q+2} + \|x\|^{p}\|h\|.
    \end{align*}
    Together with $x\in B_{R}$, this implies
    \begin{align*}
       \beta_{d_{p+1}}(x,y) \leq \tfrac{2^q}{p+1} R^{2q+2}+\tfrac{2^{3q+1}}{p+1} R^{q+1} \tau^{q+1}+\tfrac{2^{2q}}{p+1}\tau^{2q+2} + R^{p}\tau.
    \end{align*}
    For even $p$, $p=2q$ with $q\geq 1$, $x\in B_R$ and $y+h\in\E$, it follows from $\left(a^{1/t}+b^{1/t}\right)^t\leq 2^{t-1}(a+b)$ for $a,b\geq 0$ and $t\geq 1$ that
    \begin{align*}
        \beta_{d_{p+1}}(x,y) &= \tfrac{1}{p+1} \|y\|^{p+1}-\tfrac{1}{p+1} \|x\|^{p+1}- \|x\|^{p-1}\innprod{Bx}{y-x}\\
        &= \tfrac{1}{p+1} \left[\|x\|^2+\left(2\innprod{Bx}{h}+\|h\|^2\right)\right]^{\tfrac{2q+1}{2}}-\tfrac{1}{p+1} \|x\|^{p+1}- \|x\|^{p-1}\innprod{Bx}{h}\\
        &\leq \tfrac{2^{\tfrac{2q-1}{2}}}{p+1} \left[\|x\|^{2q+1}+\left(2\innprod{Bx}{h}+\|h\|^2\right)^{\tfrac{2q+1}{2}}\right]-\tfrac{1}{p+1} \|x\|^{p+1}- \|x\|^{p-1}\innprod{Bx}{h}\\
        &\leq \tfrac{2^{\tfrac{2q-1}{2}}}{p+1} \left[\|x\|^{2q+1}+2^{\tfrac{2q-1}{2}}\left(2^{\tfrac{2q+1}{2}}\innprod{Bx}{h}^{\tfrac{2q+1}{2}}+\|h\|^{2q+1}\right)\right]- \|x\|^{p-1}\innprod{Bx}{h}\\
        &\leq \tfrac{2^{\tfrac{2q-1}{2}}}{p+1} \|x\|^{2q+1}+ \tfrac{2^{\tfrac{6q-1}{2}}}{p+1} \|x\|^{\tfrac{2q+1}{2}} \|h\|^{\tfrac{2q+1}{2}}+ 2^{2q-1} \|h\|^{2q+1}+\|x\|^{p}\|h\|.
    \end{align*}
    Combining with $x\in B_{R}$, it holds that
    \begin{align*}
        \beta_{d_{p+1}}(x,y) \leq \tfrac{2^{\tfrac{2q-1}{2}}}{p+1} R^{2q+1}+ \tfrac{2^{\tfrac{6q-1}{2}}}{p+1} R^{\tfrac{2q+1}{2}} \tau^{\tfrac{2q+1}{2}}+ 2^{2q-1} \tau^{2q+1}+R^{p}\tau,
    \end{align*}
    which leads to
    \begin{equation*}
        \widehat{\theta}_R(\tau) = \left\{
        \begin{array}{ll}
           \tfrac{2^{2q}}{p+1}\tau^{2q+2} +\tfrac{2^{3q+1}}{p+1} R^{q+1} \tau^{q+1}+ R^{p}\tau+\tfrac{2^q}{p+1} R^{2q+2}  & \mathrm{if}~ p=2q+1,\vspace{2mm}\\
           2^{2q-1} \tau^{2q+1}+\tfrac{2^{\tfrac{6q-1}{2}}}{p+1} R^{\tfrac{2q+1}{2}} \tau^{\tfrac{2q+1}{2}}+ R^{p}\tau+\tfrac{2^{\tfrac{2q-1}{2}}}{p+1} R^{2q+1}  & \mathrm{if}~ p=2q.
        \end{array}
        \right.
    \end{equation*}
    To further simplify our upper bounds, for $p=2q+1$, we search for $\alpha, \beta > 1$ such that
    \begin{align*}
        a_1 \tau^{2q+2}+b_1\tau^{q+1}+c_1\tau+d_1\leq \alpha_1 a_1 \tau^{2q+2}+ \beta_1 d_1,
    \end{align*}
    or equivalently 
    \begin{align*}
        b_1+\tfrac{c_1}{\tau^q}\leq (\alpha_1-1) a_1 \tau^{q+1}+ (\beta_1-1) \tfrac{d_1}{\tau^{q+1}}.
    \end{align*}
    Now, minimizing the right-hand-side of this inequality with respect to $\tau$ leads to the optimal point $\widehat{\tau}_1=\left(\tfrac{(\beta_1-1)d_1}{(\alpha_1-1)a_1}\right)^{\tfrac{1}{2q+2}}$. Substituting this into the last inequality, we come to
    \begin{align*}
        c_1\leq \left[2\sqrt{(\alpha_1-1)(\beta_1-1)a_1d_1}-b_1\right] \left(\tfrac{(\beta_1-1)d_1}{(\alpha_1-1)a_1}\right)^{\tfrac{q}{2q+2}}.
    \end{align*}
    Let us set $2\sqrt{(\alpha_1-1)(\beta_1-1)a_1d_1}-b_1=1$ that consequently implies $c_1\leq\left(\tfrac{(\beta_1-1)d_1}{(\alpha_1-1)a_1}\right)^{\tfrac{q}{2q+2}}$ and $(\beta_1-1)=\tfrac{(b_1+1)^2}{4a_1d_1(\alpha_1-1)}$, i.e.,
    \begin{align*}
        1<\alpha_1 \leq 1+\tfrac{b_1+1}{2a_1}c_1^{-\tfrac{q+1}{q}}, \quad 1+\tfrac{b_1+1}{2d_1}c_1^{\tfrac{q+1}{q}}\leq \beta_1.
    \end{align*}
    giving \eqref{eq:dp1GradDomin} for $p=2q+1$. On the other hand, for $p=2q$, we explore the constants $\alpha_2, \beta_2 > 1$ such that the inequality
    \begin{align*}
        a_2 \tau^{2q+1}+b_2\tau^{\tfrac{2q+1}{2}}+c_2\tau+d_2\leq \alpha_2 a_2 \tau^{2q+1}+ \beta_2 d_2
    \end{align*}
    holds, which is equivalent to
    \begin{align*}
        b_2+c_2 \tau^{-\tfrac{2q-1}{2}}\leq (\alpha_2-1) a_2 \tau^{\tfrac{2q+1}{2}}+ (\beta_2-1)d_2 \tau^{-\tfrac{2q+1}{2}}.
    \end{align*}
    Let us minimize the right-hand-side of the latter inequality with respect to $\tau$ leading to the solution $\widehat{\tau}_2=\left(\tfrac{(\beta_2-1)d_2}{(\alpha_2-1)a_2}\right)^{\tfrac{1}{2q+1}}$. Now, by substituting this into point the last inequality, we get
    \begin{align*}
        c_2\leq \left[2\sqrt{(\alpha_2-1)(\beta_2-1)a_2d_2}-b_2\right] \left(\tfrac{(\beta_2-1)d_2}{(\alpha_2-1)a_2}\right)^{\tfrac{2q-1}{2q+1}}.
    \end{align*}
    Setting $\sqrt{(\alpha_2-1)(\beta_2-1)a_2d_2}-b_2=1$, it holds that
    \begin{align*}
        c_2\leq \left(\tfrac{(\beta_2-1)d_2}{(\alpha_2-1)a_2}\right)^{\tfrac{2q-1}{2q+1}}, \quad \beta_2-1=\tfrac{(b_2+1)^2}{4a_2d_2(\alpha_2-1)}
    \end{align*}
    which leads to the inequalities
    \begin{align*}
        1<\alpha_2 \leq 1+\tfrac{b_2+1}{2a_2}c_2^{-\tfrac{2q+1}{2(2q-1)}}, \quad 1+\tfrac{b_2+1}{2d_2} c_2^{\tfrac{2q+1}{2(2q-1)}}\leq \beta_2,
    \end{align*}
    giving \eqref{eq:dp1GradDomin}.
\end{proof}

Applying Lemma~\ref{lem:d4GradDomin}, we next show that $\rho_k(\cdot)$ is norm-dominated.

\begin{lem}[norm-dominatedness of the scaling function $\rho_k(\cdot)$]\label{lem:rhokGradDomin}
    Let $p\geq 2$ and $q=\floor{p/2}$. Then, the scaling function $\rho_k(\cdot)$ is norm-dominated on the Euclidean ball $B_{D_1}(y_k)=\set{x\in\E\mid \|x-y_k\|\leq D_1}$ by
    \begin{equation}\label{eq:rhokGradDomin}
        \theta_R(\tau) = \tfrac{\widehat{L}}{2} \tau^2+\left\{
        \begin{array}{ll}
           \alpha_1 a_1 \tau^{2q+2}+ \beta_1 d_1  &~~ \mathrm{if}~ p=2q+1,\vspace{2mm}\\
           \alpha_2 a_2 \tau^{2q+1}+ \beta_2 d_2  &~~ \mathrm{if}~ p=2q,
        \end{array}
        \right.
    \end{equation}
    with $\tau\geq 0$, $\widehat{L}=\sum_{k=1}^{q} \tfrac{4(1-(2D_1^2)^{2k-1})\sum_{i=1}^{2k-1}\binom{2k-1}{i}}{(1-2D_1^2)(2k-1)!}$, and $\alpha_1,\beta_1,\alpha_2,\beta_2>1$ and $a_1, d_1, a_2, d_2\geq 0$ are defined in \eqref{eq:alphaBetaRange} and \eqref{eq:abcd}, respectively.
\end{lem}

\begin{proof}
    Let us define the function $\func{\widehat{\rho}_k}{\E}{\R}$ given by
    \begin{align*}
        \widehat{\rho}_k(x) = \sum_{k=1}^{q} \tfrac{1}{(2k)!} D^{2k} f(y_k)[x-y_k]^{2k},
    \end{align*}
    where $\nabla \widehat{\rho}_k(x)=\sum_{k=1}^{q} \tfrac{1}{(2k-1)!} D^{2k} f(y_k) [x-y_k]^{2k-1}$. For $x,y\in B_{D_1}(y_k)$, we consequently have
    \begin{align*}
        \|\nabla \widehat{\rho}_k(x)&-\nabla \widehat{\rho}_k(y)\| =\left\|\sum_{k=1}^{q} \tfrac{1}{(2k-1)!} D^{2k} f(y_k)\left([x-y_k]^{2k-1}-[y-y_k]^{2k-1}\right)\right\|\\
        &\leq \sum_{k=1}^{q} \tfrac{1}{(2k-1)!} \|D^{2k} f(y_k)\| \left\|[x-y+y-y_k]^{2k-1}-[y-x+x-y_k]^{2k-1}\right\|\\
        &\leq \sum_{k=1}^{q} \tfrac{1}{(2k-1)!} \|D^{2k} f(y_k)\| \left\|\sum_{i=1}^{2k-1}\binom{2k-1}{i}[x-y]^i\left([y-y_k]^{2k-1-i}-[x-y_k]^{2k-1-i}\right)\right\|\\
        &\leq \sum_{k=1}^{q} \tfrac{1}{(2k-1)!} \|D^{2k} f(y_k)\| \left(\tfrac{2}{D_1^2}\sum_{i=1}^{2k-1}\binom{2k-1}{i} (2D_1^2)^i\right) \|x-y\|\\
        &\leq \sum_{k=1}^{q} \tfrac{1}{(2k-1)!} \|D^{2k} f(y_k)\| \left(\left(\sum_{i=1}^{2k-1}\binom{2k-1}{i}\right) \tfrac{4(1-(2D_1^2)^{2k-1})}{1-2D_1^2}\right) \|x-y\|\\
        &\leq \left(\sum_{k=1}^{q} \tfrac{4(1-(2D_1^2)^{2k-1})\sum_{i=1}^{2k-1}\binom{2k-1}{i}}{(1-2D_1^2)(2k-1)!} \right) \|x-y\|=\widehat{L} \|x-y\|,
    \end{align*}
    which means that the function $\widehat{\rho}_k(\cdot)$ is $\widehat{L}$-smooth. 
    
    From the definition of $\rho_k(\cdot)$ and the $\widehat{L}$-smoothness of $\widehat{\rho}_k(\cdot)$, we obtain
    \begin{align*}
        \beta_{\rho_k}(x,y) &= \widehat{\rho}_k(y)-\widehat{\rho}_k(x)-\innprod{\nabla \widehat{\rho}_k(x)}{y-x}\\
        &~~~+ H \left(d_{p+1}(y-y_k)-d_{p+1}(x-y_k)-\innprod{\nabla d_{p+1}(x-y_k)}{y-x} \right)\\
        &\leq \tfrac{\widehat{L}}{2} \|x-y\|^2+H \beta_{d_{p+1}}(x-y_k,y-y_k).
    \end{align*}
    Together with \eqref{eq:dp1GradDomin}, this establishes our claim.
\end{proof}

We now have all the ingredients to address the complexities of the upper and lower levels of Algorithm \ref{alg:supFP1OA}, which is the main result of this section. To this end, for the auxiliary minimization problem \eqref{eq:zk1}, we assume 
\begin{equation}\label{eq:assumMain}
    M_p(f)<+\infty,\quad R_0=\|x_0-x^*\|,\quad D_0=\max_{z\in \dom \psi} \set{\|z-x^*\| ~:~ F(z)\leq F(x_0)}<+\infty.
\end{equation}
Let us set $S=\set{z\in\dom \psi ~:~ \|z-x^*\| \leq 2 R_0}$ and assume
\begin{equation}
    F_S = \sup_{z\in S} F(z)<+\infty.
\end{equation}

\begin{thm}[complexity of Algorithm \ref{alg:supFP1OA}]
\label{thm:CompSupFP1OA}
    Let us assume that all conditions of Theorem \ref{thm:wellDefinedAlg3} hold, and let $p\geq 2$ and $q=\floor{p/2}$. Then, Algorithm \ref{alg:supFP1OA} attains an $\varepsilon$-solution of the problem \eqref{eq:prob} in
    \begin{align*}
        (2p+2) \left(\tfrac{3p M_{p+1}(f)}{(p-1)(p+1)(p-1)! \varepsilon}\right)^{\tfrac{1}{p+1}} R_0
    \end{align*}
    iterations, for the accuracy parameter $\varepsilon>0$. Moreover, the auxiliary problem \eqref{eq:zk1} is approximately solved by Algorithm \ref{alg:zk1} in at most
    \begin{equation}\label{eq:1/varepsilon}
    	1+ \tfrac{2(p+1)}{\kappa}\log\left(\tfrac{\tfrac{D(L-\mu)\overline{L}}{\beta} \left(\tfrac{2(p+1)}{\sigma}\log \theta_R(D_1)\right)^{1/(p+1)}}{ \varepsilon} \right)
    \end{equation}
iterations. 
\end{thm}

\begin{proof}
    The complexity of Algorithm~\ref{alg:supFP1OA} is a direct result of Theorem \ref{thm:aihppConv}. In order to show the second statement, let us set $R=D_1$, i.e., $\|x-y_k\|\leq D_1$ for all $x\in B_{D_1}$. From Algorithm~\ref{alg:supFP1OA}, we have $y_k= \tfrac{A_{k-1}}{A_k}x_k+\tfrac{a_k}{A_k}\upsilon_k$ and $\|\upsilon_k-x^*\| \leq 2^{\tfrac{p-1}{p+1}} \|x_0-x^*\|$.
    Using the definition of $y_k$ and Theorem~\ref{thm:aihppConv}, we come to the inequality
    \begin{align*}
        \|y_k-x^*\| &\leq \tfrac{A_k}{A_{k+1}} \|x_k-x^*\| + \tfrac{a_{k+1}}{A_{k+1}} \|\upsilon_k-x^*\| \leq \tfrac{A_k+a_{k+1}}{A_{k+1}} \max\set{\|x_k-x^*\|,\|\upsilon_k-x^*\|}\\
        &\leq \max\set{D_0,2^{\tfrac{p-1}{p+1}}R_0}.
    \end{align*}
    Invoking Theorem \ref{thm:aihppConv}(iii), it holds that
    \begin{align*}
        \|\upsilon_k-x^*\| \leq 2^{\tfrac{p-1}{p+1}} \|x_0-x^*\| \leq 2 R_0,
    \end{align*}
    i.e., $\upsilon_k\in S$. Together with the convexity of $\psi(\cdot)$ and the monotonicity of the sequence $\seq{F(x_k)}$, this implies 
    \begin{align*}
        F(y_k) \leq \tfrac{A_k}{A_{k+1}} F(x_k) + \tfrac{a_{k+1}}{A_{k+1}} F(\upsilon_k) \leq \tfrac{A_k}{A_{k+1}} F(x_0)+ \tfrac{a_{k+1}}{A_{k+1}} F_S \leq \max \set{F(x_0),F_S}.
    \end{align*}
    It follows from \eqref{eq:monotonezk} that $\varphi_k(z_i)\leq \varphi_k(z_{i-1})\leq \cdots \leq \varphi_k(z_0)=F(y_k)$. Combining with \eqref{eq:betaRhoIneq1}, \eqref{eq:parSuperfastAlg1}, and \eqref{eq:dp1LowerBound}, this implies
    \begin{align*}
        \tfrac{3H}{(p+1)2^{p-1}} \|z_i-y_k\|^{p+1} \leq F(y_k) -F(z_i) \leq \max \set{F(x_0),F_S}-F^*,
    \end{align*}
    leading to $\|z_i-y_k\|\leq \left(\tfrac{(p+1)2^{p-1}}{3H} (\max \set{F(x_0),F_S}-F^*)\right)^{\tfrac{1}{p+1}}=D_1$. Hence, these inequalities yield
    \begin{align*}
        \|z_i-x^*\|&\leq \|z_i-y_k\| + \|y_k-x^*\| 
        \leq D_1+\max\set{D_0,2^{\tfrac{p-1}{p+1}}R_0}=D.
    \end{align*}
    This implies that all conditions of Theorem~\ref{thm:wellDefinedAlg3} are satisfied. On the other hand, from the definition $\theta_R(\cdot)$ given in \eqref{eq:rhokGradDomin}, we obtain
    \begin{align*}
        \log \theta_R(\|z_k^*-y_k\|) \leq \log\theta_R(D_1).
    \end{align*}
    Then, from the uniform convexity of $\rho_k(\cdot)$ with the degree $\eta=p+1$, \eqref{eq:rhokGradDomin}, $C=\tfrac{L\sigma}{(p+1)(L-\mu)^{p+1} \overline L^{p+1}}$, and the proof of Theorem \ref{thm:wellDefinedAlg3}, we come to
    \begin{align*}\label{eq:lowerBoundi1}
        i_k^*&\leq 1+ \tfrac{p+1}{-\log\left(1-\tfrac{\kappa}{2}\right)}\log\left(\tfrac{\tfrac{D}{\beta} \left(\tfrac{2L}{C}\log \theta_R(\|z_k^*-y_k\|)\right)^{1/(p+1)}}{ \varepsilon} \right),
    \end{align*}
    which leads to \eqref{eq:1/varepsilon}. 
\end{proof}

Let us fix $q\geq 1$. Then, the function $f_{y_k,H}^p(\cdot)$ is $L$-smooth and $\mu$-strongly convex relative to the scaling function $\rho_k(\cdot)$ \eqref{eq:rhoBarxPth}, which is the same for both $p=2q$ and $p=2q+1$. If $p$ is even (i.e., $p=2q$), then Algorithm \ref{alg:supFP1OA} is a $2q$-order method (requiring the $2q$-order oracle) and attains the complexity of order $\mathcal{O}(\varepsilon^{-1/(2q+1)})$, which worse than the optimal complexity $\mathcal{O}(\varepsilon^{-2/(6q+1)})$. On the other hand, if $p$ is odd ($p=2q+1$), then Algorithm \ref{alg:supFP1OA} is again a $2q$-order method (requiring the $2q$-order oracle) and obtains the complexity of order $\mathcal{O}(\varepsilon^{-1/(2q+2)})$, which is also worse than the optimal complexity $\mathcal{O}(\varepsilon^{-2/(6q+1)})$ except for $p=3$ overpassing the classical complexity bound of second-order methods, as discussed in \cite{nesterov2020inexact}. However, in the following example, we show that the complexity of our method can overpass the classical bounds for some structured class of problems.

\begin{exa}
    Let us consider the vector $b\in\R^N$, the vectors $a_i\in\R^n$ and the univariate functions $\func{f_i}{\R}{\R}$ that are four times continuously differentiable, for $i=1,\ldots,N$. Then, we define the function $\func{f}{\R^n}{\R}$ as
    \begin{equation}
        f(x)=\sum_{i=1}^N f_i(\innprod{a_i}{x}-b_i).
    \end{equation}
    We are interested to apply Algorithm~\ref{alg:supFP1OA} with $p=4$ and $p=5$ to the problem \eqref{eq:prob} with this function $f(\cdot)$. In case of $p=5$, $q=\floor{5/2}=2$ and we need to handle the subproblem 
    \begin{equation*}\label{eq:zk1p5}
        z_{i+1} = \argmin_{z\in\E}\set{\innprod{\nabla f_{y_k,H}^5(z_i)}{z-z_i}+\psi(z)+2L \beta_{\rho_k}(z_i,z)},
    \end{equation*}
    with 
    \begin{equation*}\label{eq:rhoBarxPth1}
        \rho_k(x)= \tfrac{1}{2}\innprod{\nabla^2 f(y_k)(x-y_k)}{x-y_k} + \tfrac{1}{24}D^4 f(y_k)[x-y_k]^4+H d_6(x-y_k),
    \end{equation*}
    which readily implies that our method requires fourth-order oracle of $f_i(\cdot)$, for $i=1,\ldots,N$. Let us emphsize that Theorem~\ref{thm:nabla2Ineq} implies that the sacling function $\rho_k(\cdot)$ is convex, which  is an interesting result even in one dimension and with $N=1$, i.e.,
\begin{align*}
 f''(y_k) + \tfrac{1}{2}f^{iv}(y_k)h^2+5H|h|^4 \succeq 0.
\end{align*} 
	In the same way, for $p=4$, we need fourth-order oracle of $f_i(\cdot)$, for $i=1,\ldots,N$.    
     Moreover, Theorem~\ref{thm:CompSupFP1OA} ensures that the sequence generated by Algorithm~\ref{alg:supFP1OA} attains the complexity $\mathcal{O}(\varepsilon^{-1/5})$ for $p=4$ and $\mathcal{O}(\varepsilon^{-1/6})$ for $p=5$, which are worse that the optimal complexity $\mathcal{O}(\varepsilon^{-2/13})$, for the accuracy parameter $\varepsilon$. On the other hand, setting $h=x-y_k$, it holds that
    \begin{align*}
        \innprod{\nabla^2 f(y_k)h}{h}= \sum_{i=1}^N \nabla^2 f_i(\innprod{a_i}{y_k}-b_i) \innprod{a_i}{h}^2,\quad D^4 f(y_k)[h]^4= \sum_{i=1}^N \nabla^4 f_i(\innprod{a_i}{y_k}-b_i) \innprod{a_i}{h}^4.
    \end{align*}
    Let us particularly verify these terms for $f_i(x)=-\log(x)$ ($i=1,\ldots,N$) for $x\in (0,+\infty)$, which consequently leads to
    \begin{align*}
        \nabla^2f_i(x)=\tfrac{1}{x^2}, \quad \nabla^4f_i(x)=\tfrac{6}{x^4}=6\left(\nabla^2f_i(x)\right)^2,
    \end{align*}
    i.e.,
    \begin{align*}
        D^4 f(y_k)[h]^4 = 6\sum_{i=1}^N \left(\nabla^2f_i(\innprod{a_i}{y_k}-b_i)\right)^2 \innprod{a_i}{h}^4.
    \end{align*}
    Thus, in this case, the implementation of Algorithm~\ref{alg:supFP1OA} with $p=4$ and $p=5$ only requires the second-order oracle of $f_i(\cdot)$ ($i=1,\ldots,N$) and the first-order oracle of $\psi(\cdot)$. Therefore, we end up with a second-order method with the complexity of order $\mathcal{O}(\varepsilon^{-1/5})$ for $p=4$ and $\mathcal{O}(\varepsilon^{-1/6})$ for $p=5$, which are much faster than the second-order methods optimal bound $\mathcal{O}(\varepsilon^{-2/7})$; however, choosing the odd order $p=5$, Algorithm~\ref{alg:supFP1OA} attains a better complexity. 
\end{exa}
\section{Conclusion}\label{sec:conclusion}
In this paper, we suggest a bi-level optimization (BiOPT), a novel framework for solving convex composite minimization problems, which is a generalization of the BLUM framework given in \cite{nesterov2020inexact} and involves two levels of methodologies. In the upper level, we only assume the convexity of the objective function and design some upper-level scheme using proximal-point iterations with arbitrary order. On the other hand, in the lower level, we need to solve the proximal-point auxiliary problem inexactly by some lower-level scheme. In this step, we require some more properties of the objective function for developing efficient algorithms providing acceptable solutions for this auxiliary problem at a reasonable computational cost. The overall complexity of the method will be the product of the complexities in both levels. 

We here develop the plain $p$th-order inexact proximal-point method and its acceleration using the estimation sequence technique that, respectively, achieve the convergence rate $\mathcal{O}(k^{-p})$ and $\mathcal{O}(k^{-(p+1)})$ for the iteration counter $k$. Assuming the $L$-smoothness and $\mu$-strong convexity of the differentiable part of the proximal-point objective relative to some scaling function (for $L, \mu>0$), we design a non-Euclidean composite gradient method to inexactly solve the proximal-point problem. It turns out that this method attains the complexity $\mathcal{O}(\log \tfrac{1}{\varepsilon})$, for the accuracy parameter $\varepsilon>0$.

In the BiOPT framework, we apply the accelerated $p$th-order proximal-point algorithm in the upper level, introduce a new high-order scaling function and show that the differentiable part of the auxiliary objective is smooth and strongly convex relative to this function, and solve the auxiliary problem by a non-Euclidean composite gradient method in the lower level. We consequently come to a bi-level high-order method with the complexity of order $\mathcal{O}(\varepsilon^{-1/(p+1)})$, which overpasses the classical complexity bound of second-order methods for $p=3$, as was known from \cite{nesterov2020inexact}. In general, for $p=2$ and $p\geq 3$, the complexity of our bi-level method is sub-optimal; however, we showed  that for some class of structured problems it can overpass the classical complexity bound $\mathcal{O}(\varepsilon^{-2/(3p+1)})$.
Overall, the BiOPT framework paves the way toward methodologies using the $p$th-order proximal-point operator in the upper level and requiring lower-order oracle than $p$ in the lower level. Therefore, owing to this framework, we can design lower-order methods with convergence rates overpassing the classical complexity bounds for convex composite problems. Hence, this will open up an entirely new ground for developing novel efficient algorithms for convex composite optimization that was not possible in the classical complexity theory. 

Several extensions of our framework are possible.
As an example, we will present some extension of our framework using a segment search in the upcoming article \cite{ahookhosh2021high1}. Moreover, the proximal-point auxiliary problem can be solved by some more efficient method like the non-Euclidean Newton-type method presented in \cite{ahookhosh2019bregman}. In addition, the introduced high-order scaling function can be employed to extend the second-order methods presented in \cite{nesterov2019implementable,nesterov2019inexact,nesterov2020inexact,nesterov2020superfast} to higher-order methods.
 

  		\bibliographystyle{plain}
	\bibliography{Bibliography}

\begin{thebibliography}{10}

\bibitem{agarwal2018lower}
Naman Agarwal and Elad Hazan.
\newblock Lower bounds for higher-order convex optimization.
\newblock In {\em Conference On Learning Theory}, pages 774--792, 2018.

\bibitem{ahookhosh2019accelerated}
Masoud Ahookhosh.
\newblock Accelerated first-order methods for large-scale convex optimization:
  nearly optimal complexity under strong convexity.
\newblock {\em Mathematical Methods of Operations Research}, 89(3):319--353,
  2019.

\bibitem{ahookhosh2021high1}
Masoud Ahookhosh and Yurii Nesterov.
\newblock High-order methods beyond the classical complexity bounds, {II}:
  inexact high-order proximal-point methods with segment search.
\newblock {\em Technical report, CORE Discussion paper 2021, Universit{\'e}
  catholique de Louvain}, 2021.

\bibitem{ahookhosh2017optimal}
Masoud Ahookhosh and Arnold Neumaier.
\newblock Optimal subgradient algorithms for large-scale convex optimization in
  simple domains.
\newblock {\em Numerical Algorithms}, 76(4):1071--1097, 2017.

\bibitem{ahookhosh2018solving}
Masoud Ahookhosh and Arnold Neumaier.
\newblock Solving structured nonsmooth convex optimization with complexity
  $\mathcal{O}(\varepsilon^{-1/2})$.
\newblock {\em Top}, 26(1):110--145, 2018.

\bibitem{ahookhosh2019bregman}
Masoud Ahookhosh, Andreas Themelis, and Panagiotis Patrinos.
\newblock A {B}regman forward-backward linesearch algorithm for nonconvex
  composite optimization: superlinear convergence to nonisolated local minima.
\newblock {\em SIAM Journal on Optimization}, 31(1):653--685, 2021.

\bibitem{arjevani2019oracle}
Yossi Arjevani, Ohad Shamir, and Ron Shiff.
\newblock Oracle complexity of second-order methods for smooth convex
  optimization.
\newblock {\em Mathematical Programming}, 178(1-2):327--360, 2019.

\bibitem{baes2009estimate}
Michel Baes.
\newblock Estimate sequence methods: extensions and approximations.
\newblock {\em Institute for Operations Research, ETH, Z{\"u}rich,
  Switzerland}, 2009.

\bibitem{bauschke2016descent}
Heinz~H. Bauschke, J\'er\^ome Bolte, and Marc Teboulle.
\newblock A descent lemma beyond {L}ipschitz gradient continuity: first-order
  methods revisited and applications.
\newblock {\em Mathematics of Operations Research}, 42(2):330--348, 2016.

\bibitem{birgin2017worst}
Ernesto~G Birgin, JL~Gardenghi, Jos{\'e}~Mario Mart{\'\i}nez, Sandra~Augusta
  Santos, and Ph~L Toint.
\newblock Worst-case evaluation complexity for unconstrained nonlinear
  optimization using high-order regularized models.
\newblock {\em Mathematical Programming}, 163(1-2):359--368, 2017.

\bibitem{bolte2018first}
J\'er\^ome Bolte, Shoham Sabach, Marc Teboulle, and Yakov Vaisbourd.
\newblock First order methods beyond convexity and {L}ipschitz gradient
  continuity with applications to quadratic inverse problems.
\newblock {\em SIAM Journal on Optimization}, 28(3):2131--2151, 2018.

\bibitem{gasnikov2019optimal}
Alexander Gasnikov, Pavel Dvurechensky, Eduard Gorbunov, Evgeniya Vorontsova,
  Daniil Selikhanovych, and C{\'e}sar Uribe.
\newblock Optimal tensor methods in smooth convex and uniformly convex
  optimization.
\newblock In {\em Proceedings of the Thirty-Second Conference on Learning
  Theory}, pages 1374--1391, 2019.

\bibitem{grapiglia2020inexact}
Geovani~Nunes Grapiglia and Yu~Nesterov.
\newblock On inexact solution of auxiliary problems in tensor methods for
  convex optimization.
\newblock {\em Optimization Methods and Software}, pages 1--26, 2020.

\bibitem{guler1992newproximal}
Osman Güler.
\newblock New proximal point algorithms for convex minimization.
\newblock {\em SIAM Journal on Optimization}, 2(4):649--664, 1992.

\bibitem{iusem1994entropy}
Alfredo~N Iusem, Benar~Fux Svaiter, and Marc Teboulle.
\newblock Entropy-like proximal methods in convex programming.
\newblock {\em Mathematics of Operations Research}, 19(4):790--814, 1994.

\bibitem{jiang2019optimal}
Bo~Jiang, Haoyue Wang, and Shuzhong Zhang.
\newblock An optimal high-order tensor method for convex optimization.
\newblock In {\em Conference on Learning Theory}, pages 1799--1801, 2019.

\bibitem{lu2018relatively}
Haihao Lu, Robert~M. Freund, and Yurii Nesterov.
\newblock Relatively smooth convex optimization by first-order methods, and
  applications.
\newblock {\em SIAM Journal on Optimization}, 28(1):333--354, 2018.

\bibitem{martinet1970breve}
Bernard Martinet.
\newblock Br{\`e}ve communication. {R}{\'e}gularisation d'in{\'e}quations
  va\-ria\-tion\-nelles par approximations successives.
\newblock {\em Revue fran{\c{c}}aise d'informatique et de recherche
  op{\'e}rationnelle. S{\'e}rie rouge}, 4(R3):154--158, 1970.

\bibitem{martinet1972determination}
Bernard Martinet.
\newblock D{\'e}termination approch{\'e}e d’un point fixe d’une application
  pseudo-contractante.
\newblock {\em CR Acad. Sci. Paris}, 274(2):163--165, 1972.

\bibitem{nemirovsky1983problem}
Arkadii Nemirovsky and David Yudin.
\newblock {\em Problem Complexity and Method Efficiency in Optimization}.
\newblock John Wiley \& Sons, 1983.

\bibitem{nesterov2005smooth}
Yurii Nesterov.
\newblock Smooth minimization of non-smooth functions.
\newblock {\em Mathematical Programming}, 103(1):127--152, 2005.

\bibitem{nesterov2008accelerating}
Yurii Nesterov.
\newblock Accelerating the cubic regularization of newton’s method on convex
  problems.
\newblock {\em Mathematical Programming}, 112(1):159--181, 2008.

\bibitem{nesterov2013gradient}
Yurii Nesterov.
\newblock Gradient methods for minimizing composite functions.
\newblock {\em Mathematical Programming}, 140(1):125--161, 2013.

\bibitem{nesterov2015universal}
Yurii Nesterov.
\newblock Universal gradient methods for convex optimization problems.
\newblock {\em Mathematical Programming}, 152(1-2):381--404, 2015.

\bibitem{nesterov2018lectures}
Yurii Nesterov.
\newblock {\em Lectures on Convex Optimization}, volume 137.
\newblock Springer, 2018.

\bibitem{nesterov2019implementable}
Yurii Nesterov.
\newblock Implementable tensor methods in unconstrained convex optimization.
\newblock {\em Mathematical Programming}, pages 1--27, 2019.

\bibitem{nesterov2019inexact}
Yurii Nesterov.
\newblock Inexact basic tensor methods.
\newblock Technical report, Technical report, Technical Report CORE Discussion
  paper 2019, Universit{\'e} catholique de Louvain, 2019.

\bibitem{nesterov2020inexact}
Yurii Nesterov.
\newblock Inexact accelerated high-order proximal-point methods.
\newblock Technical report, Technical report, Technical Report CORE Discussion
  paper 2020, Universit{\'e} catholique de Louvain, 2020.

\bibitem{nesterov2020superfast}
Yurii Nesterov.
\newblock Superfast second-order methods for unconstrained convex optimization.
\newblock Technical report, Technical report, Technical Report CORE Discussion
  paper 2020, Universit{\'e} catholique de Louvain, 2020.

\bibitem{nesterov1994interior}
Yurii Nesterov and Arkadii Nemirovskii.
\newblock {\em Interior-Point Polynomial Algorithms in Convex Programming},
  volume~13.
\newblock SIAM, 1994.

\bibitem{rockafellar1976monotone}
R.~Tyrrell Rockafellar.
\newblock Monotone operators and the proximal point algorithm.
\newblock {\em SIAM Journal on Control and Optimization}, 14(5):877--898, 1976.

\bibitem{teboulle1992entropic}
Marc Teboulle.
\newblock Entropic proximal mappings with applications to nonlinear
  programming.
\newblock {\em Mathematics of Operations Research}, 17(3):670--690, 1992.

\end{thebibliography}

\end{document}